\documentclass[article,12pt]{amsart}

\usepackage{amsfonts}
\usepackage{amsmath}
\usepackage{amsthm}
\usepackage{amssymb}
\usepackage{graphicx}
\usepackage{enumerate}
\usepackage{color}
\usepackage{bbm}

\allowdisplaybreaks

\numberwithin{equation}{section}

{\theoremstyle{definition} \newtheorem{defi}{Definition}[section]}
{\theoremstyle{plain} \newtheorem{thm}[defi]{Theorem}}
{\theoremstyle{plain} }
{\theoremstyle{plain} \newtheorem{prop}[defi]{Proposition}}
{\theoremstyle{plain} \newtheorem{lem}[defi]{Lemma}}
{\theoremstyle{plain} \newtheorem{rem}[defi]{Remark}}

\newcommand{\beq} {\begin{eqnarray*}}
\newcommand{\eeq} {\end{eqnarray*}}

\newcounter{assumption}

\newcommand{\hTTn}{\widehat{\theta}_{n}}

\newcommand{\veps}{\varepsilon}

\newcommand{\dC}{\mathbb{C}}
\newcommand{\dE}{\mathbb{E}}
\newcommand{\dK}{\mathbb{K}}
\newcommand{\dN}{\mathbb{N}}
\newcommand{\dP}{\mathbb{P}}
\newcommand{\dR}{\mathbb{R}}

\newcommand{\cC}{\mathcal{C}}
\newcommand{\cD}{\mathcal{D}}

\newcommand{\cH}{\mathcal{H}}
\newcommand{\cL}{\mathcal{L}}
\newcommand{\cN}{\mathcal{N}}
\newcommand{\cP}{\mathcal{P}}
\newcommand{\cS}{\mathcal{S}}
\newcommand{\cU}{\mathcal{U}}

\newcommand{\dd}{\mathrm{d}}
\newcommand{\de}{\mathrm{e}}
\newcommand{\di}{\mathrm{i}}

\newcommand{\ind}{\mathbbm{1}}

\def \E{\mathbb{E}}
\def \Var{\hbox{{\rm Var}}}
\def \P{\mathbb{P}}

\def\leq{\leqslant}
\def\geq{\geqslant}

\newcommand{\limn}{\lim_{n\, \rightarrow\, +\infty}}
\newcommand{\cvgl}{~ \overset{\cD}{\longrightarrow} ~}
\newcommand{\cvgd}{~ \Longrightarrow ~}
\newcommand{\cvgp}{~ \overset{\dP}{\longrightarrow} ~}
\newcommand{\as}{\hspace{0.3cm} \textnormal{a.s.}}

\newcommand{\tr}{^{T}}

\newcommand{\hsp}{\hspace{0.5cm}}

\email{}


\begin{document}

\title[On the Bickel-Rosenblatt test for autoregressive processes]
{On the Bickel-Rosenblatt test of goodness-of-fit for the residuals of autoregressive processes
\vspace{2ex}}
\author[]{Agn\`es Lagnoux}
\address{Institut de Math\'ematiques de Toulouse; UMR5219. Universit\'e de Toulouse; CNRS. UT2J, F-31058 Toulouse, France.}
\email{lagnoux@univ-tlse2.fr}
\author[]{Thi Mong Ngoc Nguyen}
\address{Ho Chi Minh City University of Science, 227 Nguyen Van Cu, Phuong 4, Ho Chi Minh, Vietnam}
\email{ngtmngoc@hcmus.edu.vn}
\author[]{Fr\'ed\'eric Pro\"ia}
\address{Laboratoire Angevin de REcherche en MAth\'ematiques (LAREMA), CNRS, Universit\'e d'Angers, Universit\'e Bretagne Loire. 2 Boulevard Lavoisier, 49045 Angers cedex 01.}
\email{frederic.proia@univ-angers.fr}

\thanks{}

\begin{abstract}
We investigate in this paper a Bickel-Rosenblatt test of goodness-of-fit for the density of the noise in an autoregressive model. Since the seminal work of Bickel and Rosenblatt, it is well-known that the integrated squared error of the Parzen-Rosenblatt density estimator, once correctly renormalized, is asymptotically Gaussian for independent and identically distributed (i.i.d.)  sequences. We show that the result still holds when the statistic is built from the residuals of general stable and explosive autoregressive processes. In the univariate unstable case, we prove that the result holds when the unit root is located at $-1$ whereas we give further results when the unit root is located at $1$. In particular, we establish that except for some particular asymmetric kernels leading to a non-Gaussian limiting distribution and a slower convergence, the statistic has the same order of magnitude. We also study some common unstable cases, like the integrated seasonal process. Finally we build a goodness-of-fit Bickel-Rosenblatt test for the true density of the noise together with its empirical properties on the basis of a simulation study.
\end{abstract}

\maketitle

\section{Introduction and Motivations}
\label{SecIntro}

For i.i.d. sequences of random variables, there is a wide range of goodness-of-fit statistical procedures in connection with the underlying true distribution. Among many others, one can think about the Kolmogorov-Smirnov test, the Cram\'er-von Mises criterion, the Pearson's chi-squared test, or more specific ones like the whole class of normality tests. Most of them have become of frequent practical use and directly implemented on the set of residuals of regression models. For such applications the independence hypothesis is irrelevant, especially for time series where lagged dependent variables are included. Thus, the crucial issue that naturally arises consists in having an overview of their sensitivity facing some weakened assumptions. This paper focus on such a generalization for the \textit{Bickel-Rosenblatt} statistic, introduced by the eponymous statisticians \cite{BickelRosenblatt73} in 1973, who first established its asymptotic normality and gave their names to the associated testing procedure. The statistic is closely related to the $L^2$ distance on the real line between the Parzen-Rosenblatt kernel density estimator and a parametric distribution (or a smoothed version of it). Namely it takes the form of
\begin{equation*}
\int_{\dR} \big( \widehat{f}_{n}(x) - f(x) \big)^2 a(x)\, \dd x
\end{equation*}
with notation that we will detail in the sequel. Some improvements are of interest for us. First, Takahata and Yoshihara \cite{TakahataYoshihara87} in 1987 and later Neumann and Paparoditis \cite{NeumannPaparoditis00} in 2000 extended the result to weakly dependent sequences (with mixing or absolute regularity conditions). As they noticed, these assumptions are satisfied by several processes of the time series literature. Then, Lee and Na \cite{LeeNa02} showed in 2002 that it also holds for the residuals of an autoregressive process of order 1 as soon as it contains no unit root (we will explain precisely this fact in the paper). Such a study leads to a goodness-of-fit test for the distribution of the innovations of the process. Bachmann and Dette \cite{BachmannDette05} went further in 2005  by putting the results obtained by Lee and Na into practice. Their study also enables to get an asymptotic normality of the correctly renormalized statistic under some fixed alternatives. Even if it is not directly related to our subject, we also mention the work of Horv\'ath and Zitikis \cite{HorvathZitikis04} in 2004, providing some results on goodness-of-fit tests in $L^r$ norm (for any $r$ greater than 1) for the residuals of first-order autoregressive processes.

The goal of this paper is twofold. First and mainly, we generalize the
results of Lee and Na to autoregressive processes of order $p$ ($p \geq 1$) while refining the set of hypotheses and discussing on the effect of unit roots on the statistic of interest. Second, a goodness-of-fit test is derived as a corollary together with a short empirical study so as to get an overview of its capabilities. From a theoretical point of view, this is always a challenging target to extend well-known first-order results to more general cases, and it is even more difficult to deal with unit roots in a time series context. Moreover, autoregressive models are still widespread in applied fields like econometrics, mathematical finance, weather and energy forecasting, engineering, etc. Thus to provide some advances in the study of autoregressive processes, in terms of inference, prediction, statistical significance, or, in our case, goodness-of-fit, was our first motivation.

On finite samples, it has been observed that the Gaussian behavior is difficult to reach and that, instead, an asymmetry occurs for dependent frameworks (see, \textit{e.g.}, Valeinis and Locmelis \cite{ValeinisLocmelis12}). In the simulation study, we will use the configurations suggested in this previous paper, and that of Fan \cite{Fan94} and Ghosh and Huang \cite{GhoshHuang91} to try to minimize this effect. In the end of this section, we introduce the context of our study and present both notation and vocabulary used in the sequel. Moreover, we recall the well-known asymptotic behavior of the Bickel-Rosenblatt statistic for i.i.d. random variables. Section 3 is dedicated to our results that are proved in Appendix B. 
To sum up, we establish the asymptotic behavior of the Bickel-Rosenblatt statistic based on the residuals of stable and explosive autoregressive  processes of order $p$ ($p \geq 1$). In addition, we also prove results related to common unstable autoregressions, like the random walks and the seasonally integrated processes. Finally, we give some considerations when dealing with general unstable processes or mixed processes. For example, the unstable ARIMA($p-1$,1,0) process would deserve a particular attention due to its widespread use in the econometric field.
In Section 4, we build a goodness-of-fit test and discuss some empirical bases. The Appendix A is devoted to an overview of existing results on the least-squares estimation of the autoregressive parameter, depending on the roots of its characteristic polynomial, since it is of crucial interest in our proofs.

To start with, let us consider an autoregressive process of order $p$ (AR($p$)) defined by
\begin{equation}
\label{AR}
X_{t} = \theta_1\, X_{t-1} + \hdots + \theta_{p}\, X_{t-p} + \veps_t
\end{equation}
for any $t \geq 1$ or equivalently, in a compact form, by
\begin{equation*}
X_{t} = \theta\tr \Phi_{t-1} + \veps_t
\end{equation*}
where $\theta = (\theta_1, \hdots, \theta_{p})\tr$ is a vector parameter, $\Phi_0$ is an arbitrary initial random vector, $\Phi_{t} = (X_{t}, \hdots, X_{t-p+1} )\tr$ and $(\veps_t)$ is a strong white noise having a finite positive variance $\sigma^2$ and a marginal density $f$ (positive on the real line). The corresponding characteristic polynomial is defined, for all $z \in \dC$, by
\begin{equation}
\label{PolCar}
\Theta(z) = 1 - \theta_1\, z - \hdots - \theta_{p}\, z^{p}
\end{equation}
and the companion matrix associated with $\Theta$ (see, \textit{e.g.}, \cite[Sec. 4.1.2]{Duflo97}) is given by
\begin{equation}
\label{CompMat}
C_{\theta} = \begin{pmatrix}
\theta_1 & \theta_2 & \hdots & \theta_{p-1} & \theta_p \\
1 & 0 & \hdots & 0 & 0 \\
0 & 1 & \hdots & 0 & 0 \\
\vdots & \vdots & \ddots & \vdots & \vdots \\
0 & 0 & \hdots & 1 & 0
\end{pmatrix}.
\end{equation}
It follows that the process may also be written as
\begin{equation}
\label{VAR}
\Phi_{t} = C_{\theta}\, \Phi_{t-1} + E_{t}
\end{equation}
where $E_{t} = (\veps_{t}, 0, \hdots, 0)\tr$ is a $p$--dimensional noise. It is well-known that the stability of this $p$--dimensional process is closely related to the eigenvalues of the companion matrix that we will denote and arrange like
$$
\rho(C_{\theta}) = \vert \lambda_1 \vert \geq \vert \lambda_2 \vert \geq \hdots \geq \vert \lambda_{p} \vert.
$$
In particular, according to \cite[Def. 2.3.17]{Duflo97}, the process is said to be \textit{stable} when $\vert \lambda_1 \vert < 1$, \textit{purely explosive} when $\vert \lambda_{p} \vert > 1$ and \textit{purely unstable} when $\vert \lambda_1 \vert = \vert \lambda_{p} \vert = 1$. Among the purely unstable processes of interest, let us mention the seasonal model admitting the complex $s$--th roots of unity as solutions of its autoregressive polynomial $\Theta(z) = 1 - z^{s}$ for a season $s \in \dN\backslash\{ 0, 1\}$. In the paper, this model will be shortened as \textit{seasonal unstable of order $p=s$}, it satisfies $\theta_1 = \hdots = \theta_{s-1} = 0$ and $\theta_{s} = 1$. In a general way, it is easy to see that $\det(C_{\theta}) = (-1)^{p+1}\, \theta_{p}$ so that $C_{\theta}$ is invertible as soon as $\theta_{p} \neq 0$ (which will be one of our hypotheses when $p > 0$). In addition, a simple calculation shows that
\begin{equation*}
\det(C_{\theta} - \lambda\, I_{p}) = (-\lambda)^{p}\, \Theta(\lambda^{-1})
\end{equation*}
which implies (since $\Theta(0) \neq 0$) that each zero of $\Theta$ is the inverse of an eigenvalue of $C_{\theta}$. Consequently, the stability of the process may be expressed in the paper through the eigenvalues of $C_{\theta}$ as well as through the zeroes of $\Theta$. We will also consider that
\begin{equation*}
z^{p}\, \Theta(z^{-1}) = z^{p} - \theta_1\, z^{p-1} - \ldots - \theta_{p}
\end{equation*}
is the minimal polynomial of $C_{\theta}$ (which, in the terminology of \cite{Duflo97}, means that the process is \textit{regular}). Now, assume that $X_{-p+1}, \hdots, X_0, X_1, \hdots, X_{n}$ are observable (for $n \gg p$) and let 
\begin{equation}
\label{OLS}
\hTTn = \left(\sum_{t=1}^n \Phi_{t-1} \Phi_{t-1}\tr \right)^{\!-1} \sum_{t=1}^n \Phi_{t-1}\, X_{t}
\end{equation}
be the least-squares estimator of $\theta$ (for $p>0$). The associated residual process is
\begin{equation}
\label{ResSet}
\widehat{\veps}_{t} = X_{t} - \hTTn^{\: T}\, \Phi_{t-1}
\end{equation}
for all $1 \leq t \leq n$, or simply $\widehat{\veps}_{t} = X_{t}$ when $p=0$. Hereafter, $\dK$ is a kernel and $(h_{n})$ is a bandwidth. That is, $\dK$ is a non-negative function satisfying
\begin{equation*}
\int_{\dR} \dK(x)\, \dd x = 1, \hsp \int_{\dR} \dK^2(x)\, \dd x < +\infty \hsp \text{and} \hsp \int_{\dR} x^2\, \dK(x)\, \dd x < +\infty,
\end{equation*}
and $(h_{n})$ is a positive sequence decreasing to 0. The so-called Parzen-Rosenblatt estimator \cite{Parzen62,Rosenblatt56} of the density $f$ is given, for all $x \in \dR$, by
\begin{equation}
\label{PREst}
\widehat{f}_{n}(x) = \frac{1}{n\, h_{n}}\, \sum_{t=1}^n \dK\left( \frac{x - \widehat{\veps}_{t}}{h_{n}} \right).
\end{equation}
The local and global behaviors of this empirical density have been well studied in the literature. However, for a goodness-of-fit test, we focus on the global fitness of $\widehat{f}_{n}$ to $f$ on the whole real line. From this viewpoint, we consider the Bickel-Rosenblatt statistic that we define as
\begin{equation}
\label{BRStat}
\widehat{T}_{n} = n\, h_{n} \int_{\dR} \big( \widehat{f}_{n}(x) - (\dK_{h_{n}} * f)(x) \big)^2 a(x)\, \dd x
\end{equation}
where $\dK_{h_{n}} = h_{n}^{-1}\, \dK(\cdot/h_{n})$, $a$ is a positive piecewise continuous integrable function and $*$ denotes the convolution operator, \textit{i.e.} $(g*h)(x) = \int_{\dR} g(x-u)\, h(u)\, \dd u$. A statistic of probably greater interest and easier to implement is
\begin{equation}
\label{BRStatF0}
\widetilde{T}_{n} = n\, h_{n} \int_{\dR} \big( \widehat{f}_{n}(x) - f(x) \big)^2 a(x)\, \dd x.
\end{equation}
Bickel and Rosenblatt show in \cite{BickelRosenblatt73} that under appropriate conditions, if $T_{n}$ is the statistic given in \eqref{BRStat} built on the strong white noise $(\veps_{t})$ instead of the residuals, then, as $n$ tends to infinity,
\begin{equation}
\label{AsNormBR}
\frac{T_n - \mu}{\sqrt{h_{n}}} \cvgl \cN(0, \tau^2)
\end{equation}
where the centering term is
\begin{equation}
\label{CenterBR}
\mu = \int_{\dR} f(s)\, a(s)\, \dd s\, \int_{\dR} \dK^2(s)\, \dd s
\end{equation}
and the asymptotic variance is given by
\begin{equation}
\label{AsVarBR}
\tau^2 = 2\, \int_{\dR} f^{\, 2}(s)\, a^2(s)\, \dd s\, \int_{\dR} \left( \int_{\dR} \dK(t)\, \dK(t+s)\, \dd t \right)^{\! 2}\, \dd s.
\end{equation}
The aforementioned conditions are summarized in \cite[Sec. 2]{GhoshHuang91}, they come from the original work of Bickel and Rosenblatt later improved by Rosenblatt \cite{Rosenblatt75}. In addition to some technical assumptions that will be recalled in \ref{hyp:H0} and similarly to the results in \cite{LeeNa02}, notice that the results hold 
\begin{itemize}
\item either if $\dK$ is bounded and has a compact support and the bandwidth is given by $h_{n} = h_0\, n^{-\kappa}$ with $0 < \kappa < 1$;
\item or if $\dK$ is a continuous and positive kernel defined on $\dR$, the bandwidth is given by $h_{n} = h_0\, n^{-\kappa}$ with $0 < \kappa < 1/4$ and 
\begin{equation*}
\int_{\vert x \vert\, \geq\, 3}  \vert x \vert^{3/2}\, (\ln \ln \vert x \vert)^{1/2}\, \vert \dK^{\prime}(x) \vert\, \dd x < +\infty \hsp \text{and} \hsp \int_{\dR} (\dK^{\prime}(x))^2\, \dd x < +\infty.
\end{equation*}
\end{itemize}
We can find in later references (like \cite{TakahataYoshihara87}, \cite{NeumannPaparoditis00} or \cite{BachmannDette05}) some alternative proofs of the asymptotic normality \eqref{AsNormBR} with $a(x)=1$ and appropriate assumptions. However, in this paper, we keep $a$ as an integrable function in order to follow the original framework of Bickel and Rosenblatt. Fortunately, it also  makes the calculations easier and remains appropriate for applications since it is always possible to define a compact support including any mass of a given density $f_0$ to be tested, as it will be done in Section \ref{SecTest}.

\medskip

%\begin{rem}
\textbf{Notation.} It is convenient to have short expressions for terms that converge in probability to zero. In the whole study, the notation $o_{\P}(1)$ (resp. $O_{\P}(1)$) stands for a sequence of random variables that converges to zero in probability (resp. is bounded in probability) as $n \to \infty$.
%\end{rem}

\section{The Bickel-Rosenblatt statistic}
\label{SecBR}

In this section, we derive the limiting distribution of the test statistics $\widehat{T}_{n}$ and $\widetilde{T}_{n}$ given by \eqref{BRStat} and \eqref{BRStatF0}, based on the residuals in the stable, some unstable and purely explosive cases.

\subsection{Assumptions}

For the whole study, we make the following assumptions.
\renewcommand{\theenumi}{(H\arabic{enumi})}
\renewcommand{\labelenumi}{\theenumi}

\begin{enumerate}
\addtocounter{enumi}{-1}
\item \label{hyp:H0} The strong white noise $(\veps_{t})$ has a bounded density $f$ which is positive, twice differentiable, and the second derivative $f^{\prime \prime}$ is itself bounded. The weighting function $a$ is positive, piecewise continuous and integrable. The kernel $\dK$ is bounded, continuous on its support. The kernel $\dK$ and the bandwidth $(h_{n})$ are chosen so that 
\begin{itemize}
\item either a) $\dK$ is bounded and has a compact support and the bandwidth is given by $h_{n} = h_0\, n^{-\kappa}$ with $0 < \kappa < 1$;
\item or b) $\dK$ is continuous and positive kernel defined on $\dR$, the bandwidth is given by $h_{n} = h_0\, n^{-\kappa}$ with $0 < \kappa < 1/4$ and 
\begin{equation*}
\int_{\vert x \vert\, \geq\, 3}  \vert x \vert^{3/2}\, (\ln \ln \vert x \vert)^{1/2}\, \vert \dK^{\prime}(x) \vert\, \dd x < +\infty \hsp \text{and} \hsp \int_{\dR} (\dK^{\prime}(x))^2\, \dd x < +\infty.
\end{equation*}
\end{itemize}

\setcounter{assumption}{\arabic{enumi}}
\end{enumerate}
Some additional hypotheses are given below, not simultaneously needed.
\begin{enumerate}
\setcounter{enumi}{\theassumption}
\item \label{hyp:H1} The kernel $\dK$ is such that $\dK^{\prime \prime \prime}$ exists and is bounded on its support,
\begin{equation*}
\int_{\dR} \big\{ \vert \dK^{\prime}(x) \vert + \vert \dK^{\prime \prime}(x) \vert \big\}\, \dd x < +\infty, \hsp \int_{\dR} ( \dK^{\prime \prime}(x) )^2\, \dd x < +\infty,
\end{equation*}
\begin{equation*}
\int_{\dR} \vert x\, \dK^{\prime}(x) \vert\, \dd x < +\infty \hsp \text{and} \hsp \int_{\dR} x^2\, \vert \dK^{\prime}(x) \vert\, \dd x < +\infty.
\end{equation*}
\item \label{hyp:H2} The kernel $\dK$ satisfies
\begin{equation*}
\int_{\dR} \vert \dK(x+\delta) - \dK(x) \vert\, \dd x \leq B\, \delta
\end{equation*}
for some $B > 0$ and all $\delta >0$.
\item \label{hyp:Halpha} For some $\alpha > 0$, the bandwidth $(h_{n})$ satisfies
\begin{equation*}
\limn n\, h_{n}^{\alpha} = +\infty.
\end{equation*}
\item \label{hyp:Hbeta} For some $\beta > 0$, the bandwidth $(h_{n})$ satisfies
\begin{equation*}
\limn n\, h_{n}^{\beta} = 0.
\end{equation*}
\item \label{hyp:Hbruit} The noise $(\veps_{t})$ and the initial vector $\Phi_0$ have a finite moment of order $\nu > 0$.
\renewcommand{\theenumi}{\alph{enumi}}
\renewcommand{\labelenumi}{\theenumi.}
\end{enumerate}
Notice that  
\begin{equation*}
\int_{\dR} \vert \dK'(x) \vert\, \dd x < +\infty
\end{equation*}
is a sufficient condition to get \ref{hyp:H2}.

\subsection{Asymptotic behaviors}

\begin{thm}[Stable case]
\label{ThmStable}
In the stable case ($\vert \lambda_1 \vert < 1$), assume that \ref{hyp:H0},  \ref{hyp:Halpha} with $\alpha = 4$, and  \ref{hyp:Hbruit} with $\nu = 4$ hold. Then, as $n$ tends to infinity,
\begin{equation*}
\frac{\widehat{T}_n - \mu}{\sqrt{h_{n}}} \cvgl \cN(0, \tau^2)
\end{equation*}
where $\mu$ and $\tau^2$ are given in \eqref{CenterBR} and \eqref{AsVarBR}, respectively. In addition, the result is still valid for $\widetilde{T}_{n}$ if  \ref{hyp:Hbeta} with $\beta = 9/2$ holds.
\end{thm}
\begin{proof}
See Section \ref{SecProofStable}.
\end{proof}

\begin{rem}
Theorem \ref{ThmStable} holds for $h_{n} = h_0\, n^{-\kappa}$ as soon as $2/9 < \kappa < 1/4$. A standard choice may be $h_{n} = h_0\, n^{-1/4+\epsilon}$ for a small $\epsilon > 0$.
\end{rem}

\begin{thm}[Purely explosive case]
\label{ThmExplo}
In the purely explosive case ($\vert \lambda_{p} \vert > 1$), assume that  \ref{hyp:H0},  \ref{hyp:H2},  \ref{hyp:Halpha} with $\alpha=1$ and  \ref{hyp:Hbruit} with $\nu = 2+\gamma$ hold, for some $\gamma > 0$. Then, as $n$ tends to infinity,
\begin{equation*}
\frac{\widehat{T}_n - \mu}{\sqrt{h_{n}}} \cvgl \cN(0, \tau^2)
\end{equation*}
where $\mu$ and $\tau^2$ are given in \eqref{CenterBR} and \eqref{AsVarBR}, respectively. In addition, the result is still valid for $\widetilde{T}_{n}$ if  \ref{hyp:Hbeta} with $\beta = 9/2$ holds.
\end{thm}
\begin{proof}
See Section \ref{SecProofExplo}.
\end{proof}

\begin{rem}
Hence by \ref{hyp:H0}, Theorem \ref{ThmExplo} holds for $h_{n} = h_0\, n^{-\kappa}$ as soon as $2/9 < \kappa < 1$ if $\dK$ has a compact support (the usual bandwidth $h_{n} = h_0\, n^{-1/4}$ is thus appropriate) or  as soon as $2/9 < \kappa < 1/4$ for a kernel positive on $\dR$. In that case, a standard choice may be $h_{n} = h_0\, n^{-1/4+\epsilon}$ for a small $\epsilon > 0$.
\end{rem}

We now investigate the asymptotic behavior of the statistics for some unstable cases. We establish in particular that the analysis of \cite{LeeNa02} is only partially true.
\begin{prop}[Purely unstable case]
\label{PropUnivUnstable}
In the unstable case for $p=1$ ($\lambda = \pm 1$), assume that  \ref{hyp:H0},  \ref{hyp:H1},  \ref{hyp:Halpha} with $\alpha = 4$, and  \ref{hyp:Hbruit} with $\nu = 4$ hold. If $\lambda=- 1$ then, as $n$ tends to infinity,
% or $\lambda=1$ and $\int_{\dR} \dK^{\prime}(s)\, \dd s = 0$ then, as $n$ tends to infinity,
\begin{equation*}
\frac{\widehat{T}_n - \mu}{\sqrt{h_{n}}} \cvgl \cN(0, \tau^2)
\end{equation*}
where $\mu$ and $\tau^2$ are given in \eqref{CenterBR} and \eqref{AsVarBR}, respectively. If $\lambda=1$ and $\int_{\dR} \dK^{\prime}(s)\, \dd s = 0$ then
\begin{equation*}
\frac{\widehat{T}_n - \mu}{\sqrt{h_{n}}} = O_{\P}(1).
\end{equation*}
The results are still valid for $\widetilde{T}_{n}$ if  \ref{hyp:Hbeta} with $\beta = 9/2$ holds. Finally, if $\lambda=1$ and $\int_{\dR} \dK^{\prime}(s)\, \dd s \neq 0$ then, as $n$ tends to infinity,
\begin{equation*}
h_{n} (\widehat{T}_{n} - \mu) \cvgl \sigma^2 \left( \frac{\frac{1}{2}\, (W^2(1)-1)}{\int_0^1 W^2(u)\, \dd u}\, \int_0^1 W(u)\, \dd u \right)^{\! 2}\! \int_{\dR} f^2(s)\, a(s)\, \dd s \left( \int_{\dR} \dK^{\prime}(s)\, \dd s\right)^{\! 2}
\end{equation*}
where $(W(t),\, t \in [0,1])$ is a standard Wiener process. 

In the seasonal case with $p = s$, we also reach
\begin{equation*}
\frac{\widehat{T}_n - \mu}{\sqrt{h_{n}}} = O_{\P}(1)
\end{equation*}
or
\begin{equation*}
h_{n} (\widehat{T}_{n} - \mu) \cvgl \sigma^2 \big( S(W_{s})^{T}\, H(W_{s}) \big)^2\! \int_{\dR} f^2(s)\, a(s)\, \dd s \left( \int_{\dR} \dK^{\prime}(s)\, \dd s\right)^{\! 2}
\end{equation*}
depending on whether $\int_{\dR} \dK^{\prime}(s)\, \dd s = 0$ or $\int_{\dR} \dK^{\prime}(s)\, \dd s \neq 0$, respectively, where $S(W_{s})$ is defined in Proposition \ref{PropOLSUnstable}, $H(W_{s})$ will be clarified in the proof and $(W_{s}(t),\, t \in [0,1])$ is a standard Wiener process of dimension $s$.
\end{prop}
\begin{proof}
See Section \ref{SecProofUnstable}.
\end{proof}

\begin{rem}
This last result needs some observations.
\begin{itemize}
\item Proposition \ref{PropUnivUnstable} holds for $h_{n} = h_0\, n^{-\kappa}$ as soon as $2/9 < \kappa < 1/4$ \and $\lambda=-1$. A standard choice may be $h_{n} = h_0\, n^{-1/4+\epsilon}$ for a small $\epsilon > 0$.
\item One can observe that the value of $\int_{\dR} \dK^{\prime}(s)\, \dd s$ is crucial to deal with the unstable case. In fact, all usual kernels are even, leading to $\int_{\dR} \dK^{\prime}(s)\, \dd s = 0$. 
On simulations, the unstable case gives results similar to the stable and explosive ones, except for some very particular asymmetric kernels having different bound values that we can build to violate the 	aforementioned natural condition. In that case, the convergence is slower as we can see from the result of Proposition \ref{PropUnivUnstable}, and the limiting distribution is not Gaussian.
\item At the end of the associated proof, we show that, for $\lambda=1$ and a choice of kernel such that $\int_{\dR} \dK^{\prime}(s)\, \dd s \neq 0$, we also need $\int_{\dR} s\, \dK(s)\, \dd s = 0$ to translate the result from $\widehat{T}_{n}$ to $\widetilde{T}_{n}$ under  \ref{hyp:Hbeta} with $\beta = 6$. In this case, the result holds for $h_{n} = h_0\, n^{-\kappa}$ as soon as $1/6 < \kappa < 1/4$.
\item In \cite[Thm. 3.2]{LeeWei99}, Lee and Wei had already noticed that only the unit roots located at 1 affect the residual empirical process defined, for $0 \leq u \leq 1$, by
\begin{equation*}
\widehat{H}_{n}(u) = \frac{1}{\sqrt{n}} \sum_{t=1}^{n} \big( \ind_{\{ F(\widehat{\veps}_{t})\, \leq\, u \}} - u \big),
\end{equation*}
where $F$ is the cumulative distribution function of the noise. Our result is consistent with that fact since one may rewrite 
$\widehat{T}_{n}$ in terms of $\widehat{H}_{n}$ as
\begin{equation*}
\widehat{T}_{n} = n\, h_{n}\, \int_{\dR} \left( \frac{1}{h_{n}\, \sqrt{n}} \int_{\dR} \dK\left( \frac{x-s}{h_{n}} \right) \dd \widehat{H}_{n}(F(s)) \right)^{\, 2} a(x)\, \dd x.
\end{equation*}
Thus, $\widehat{H}_{n}$ appearing as the only stochastic part of the statistic, it is not surprising that, similarly, only the unit roots located at 1 affect our results.
\end{itemize}
\end{rem}

\begin{rem}[Beyond the results of Lee and Na] We conclude this section by a short comparison with the results of Lee and Na \cite{LeeNa02}. The set of hypotheses is refined but in the stable and explosive cases, the proofs follow similar strategies to extend their results from first-order $(p=1)$ to general autoregressive processes $(p \geq 1)$ and to the more natural unsmoothed statistic \eqref{BRStatF0}. However a substantial improvement lies in the unstable case, since this work is not only a matter of generalization of the known results, but also a matter of correction of them. In particular, we have established that the negative unit root, 	although behind instability, does not play any role. In addition, for a positive unit root, we have shown that the rate $h_{n}^{-1}$ is not systematic but only reached at the cost of the unrealistic hypothese $\int_{\dR} \dK^{\prime}(s)\, \dd s = 0$, and in this case we have provided the limiting distribution. Otherwise the order of magnitude of the statistic remains $h_{n}^{1/2}$ even if the limiting distribution (if it exists) is not still established.
\end{rem}

The authors of \cite{HorvathZitikis04} mentioned that the statistic has a rate of $n^{1/2}\, h_{n}^{(1-1/r)/2}$ when built on the residuals of first-order autoregressive processes, for a goodness-of-fit test based on the $L^{r}$ norm (which indeed corresponds to Bickel-Rosenblatt for $r=2$). Even if it is hardly comparable since it relies on the sample cumulative distribution and not on any density estimation, let us also recall that the Kolmogorov-Smirnov test, probably the most famous and popular non-parametric test of goodness-of-fit, is associated with an $L^{\infty}$ norm and has a rate of $\sqrt{n}$. The generalization of this work to these norms could also be valuable, by way of comparison.

To conclude this part, we draw the reader's attention to the fact that the purely unstable case is not fully treated. The general results may be a challenging study due to the phenomenon of compensation arising through unit roots different from 1. Lemma \ref{LemUnstable2} at the end of the Appendix B is not used as part of this paper but  may be a trail for future studies. It illustrates the compensation \textit{via} the fact that $(\sum_{k=1}^{t} X_{k})$ is of the same order as $(X_{t})$ in a purely unstable process having no unit root located at 1. Mixed models ($\vert \lambda_1 \vert \geq 1$ and $\vert \lambda_{p} \vert \leq 1$) should lead to similar reasonings, they also have to be handled to definitively lift the veil on the Bickel-Rosenblatt statistic for the residuals of autoregressive processes. As a priority, it seems that unstable ARIMA($p-1$,1,0) processes would deserve close investigations due to their widespread use in the econometric field. The difficulty arising here is that the estimator converges at the slow rate of stability while the process grows at the fast rate of instability: a compensation will be needed. Another technical improvement that would require an 	adjustment of the proofs, and probably a strengthening of the hypotheses on $\dK$ and $f$, could be the extension of the results to the weighting function $a(x) = 1$. The goodness-of-fit test could then be implemented on the whole real line and not only on a part of it, however big it is, as it will be the case in the next section.

\section{A goodness-of-fit Bickel-Rosenblatt test}
\label{SecTest}

Our objective is now to derive a goodness-of-fit testing procedure from the results established in the previous section. First, one can notice that the choice of
\begin{equation*}
a(x) = \left\{
\begin{array}{ll}
1/f(x) & \mbox{for } x \in [-\delta\,;\,\delta] \\
0 & \mbox{for } x \notin [-\delta\,;\,\delta]
\end{array}
\right.
\end{equation*}
for any $\delta > 0$ leads to the simplifications
\begin{equation*}
\widetilde{T}_{n} = n\, h_{n} \int_{-\delta}^{\delta} \frac{\big( \widehat{f}_{n}(x) - f(x) \big)^2}{f(x)}\, \dd x, \hsp \mu = 2\, \delta \int_{\dR} \dK^2(s)\, \dd s
\end{equation*}
and
\begin{equation*}
\tau^2 = 4\, \delta \int_{\dR} \left( \int_{\dR} \dK(t)\, \dK(t+s)\, \dd t \right)^{\! 2}\, \dd s.
\end{equation*}
Bickel and Rosenblatt \cite{BickelRosenblatt73} suggested a similar choice for the weight function $a$, with $[0\,;\,1]$ for compact support. Nevertheless, it seems more reasonable to work on a symmetric interval in order to test for the density of a random noise. In addition, $\mu$ and $\tau^2$ become independent of $f$, which will be useful to build a statistical procedure based on $f$. For a couple of densities $f$ and $f_0$ such that $f_0$ does not cancel on $[-\delta\,;\,\delta]$, let us define
\begin{equation}
\label{DistDens}
\Delta_{\delta}(f,f_0) = \int_{-\delta}^{\delta} \frac{\big( f(x) - f_0(x) \big)^2}{f_0(x)}\, \dd x.
\end{equation}
Hence, $\Delta_{\delta}(f,f_0) = 0$ means that $f$ and $f_0$ coincide almost everywhere on $[-\delta\,;\,\delta]$, and everywhere under our usual continuity hypotheses on the densities. On the contrary, $\Delta_{\delta}(f,f_0) > 0$ means that there exists an interval $I \subseteq [-\delta\,;\,\delta]$ with non-empty interior on which $f$ and $f_0$ differ. Accordingly, let
\begin{equation*}
\cH_0: ``\Delta_{\delta}(f,f_0) = 0" \quad \text{vs.} \quad \cH_1: ``\Delta_{\delta}(f,f_0) > 0".
\end{equation*}
The natural test statistic is therefore given by
\begin{equation}
\label{StatTest}
\widetilde{Z}_{n}^{\, 0} = \frac{\widetilde{T}_{n}^{\, 0} - \mu}{\tau\ \sqrt{h_{n}}}
\end{equation}
where $\widetilde{T}_{n}^{\, 0}$ is the statistic $\widetilde{T}_{n}$ reminded above built using $f_0$ instead of $f$.

\begin{prop}[A goodness-of-fit test]
\label{PropGof}
Consider the set of residuals from one of the following AR$(p)$ models:
\begin{itemize}
\item a stable process with $p \geq 1$, under the hypotheses of Theorem \ref{ThmStable},
\item an explosive process with $p \geq 1$, under the hypotheses of Theorem \ref{ThmExplo},
\item an unstable process with $p = 1$ and $\lambda=-1$, under the hypotheses of Proposition \ref{PropUnivUnstable}.
\end{itemize}
Then, under $\cH_0: ``\Delta_{\delta}(f,f_0) = 0"$ where $f_0$ is a continuous density which does not cancel on $[-\delta\,;\,\delta]$ for some $\delta>0$,
\begin{equation*}
\widetilde{Z}_{n}^{\, 0} \cvgl \cN(0,1).
\end{equation*}
In addition, under $\cH_1: ``\Delta_{\delta}(f,f_0) > 0"$,
\begin{equation*}
\widetilde{Z}_{n}^{\, 0} \cvgp +\infty.
\end{equation*}
\end{prop}
\begin{proof}
The proof is immediate using our previous results. The consistency under $\cH_1$ is reached using the fact that
\begin{equation*}
\widetilde{Z}_{n}^{\, 0} = \frac{\widetilde{T}_{n} - \mu}{\tau\ \sqrt{h_{n}}} + \frac{\widetilde{T}_{n}^{\, 0} - \widetilde{T}_{n}}{\tau\ \sqrt{h_{n}}}.
\end{equation*}
\end{proof}
For any level $0 < \alpha < 1$, we reject $\cH_0$ as soon as
\begin{equation*}
\widetilde{Z}_{n}^{\, 0} > u_{1-\alpha}
\end{equation*}
where $u_{1-\alpha}$ stands for the $(1-\alpha)$--quantile of the $\cN(0,1)$ distribution. For our simulations, we focus on a normality test (probably the most useful in regression, for goodness-of-fit). We have trusted the observations of \cite{GhoshHuang91}, \cite{Fan94} or \cite{ValeinisLocmelis12}. In particular, only the $\cN(0,1)$ and the $\cU([-1\,;\,1])$ kernels are used, with obviously $\int_{\dR} \dK^{\prime}(s)\, \dd s = 0$. The bandwidth is $h_{n} = h_0\, n^{-1/4+\epsilon}$ for $\epsilon = 10^{-3}$ where $h_0$ is calibrated to reach an empirical level close to $\alpha=5\%$ for the neutral model ($p=0$), and the true distribution of the noise is $\cN(0,1)$. We shall note at this point that, despite appearances, the choice of $\kappa$ is not crucial for our study. Indeed, for the moderate values of $n$ that we consider, $h_0$ plays a more important role. That is the reason why we choose $\kappa$ in the restricted area of validity $(2/9 < \kappa < 1/4)$ and why we focus first on $h_0$. The selection of the hyperparameter $\delta$ will be described thereafter. We only give an overview of the results in Table \ref{TabH0} for some typical models:
\begin{itemize}
\item M$_0$ -- neutral model ($p=0$),
\item M$_1$ -- stable model ($p=3$, $\theta_1 = -{1}/{12}$, $\theta_2 = {5}/{24}$, $\theta_3 = {1}/{24}$),
\item M$_2$ -- stable but almost unstable model ($p=1$, $\theta_1 = {99}/{100}$),
\item M$_3$ -- unstable model with negative unit root ($p=1$, $\theta_1 = -1$),
\item M$_4$ -- unstable model with positive unit root ($p=1$, $\theta_1 = 1$),
\item M$_5$ -- explosive model ($p=2$, $\theta_1 = 0$, $\theta_2 = {121}/{100}$).
\end{itemize}

\begin{table}[h!]
\centering
\begin{tabular}{|c||ccc||ccc|} 
\hline
$\dK$ & \multicolumn{3}{|c||}{$\cN(0,1)$} & \multicolumn{3}{|c|}{$\cU([-1\,;\,1])$} \\
\hline
$n$ & 50 & 100 & 500 & 50 & 100 & 500 \\
\hline
$h_0$ & 0.10 & 0.14 & 0.14 & 0.20 & 0.25 & 0.32 \\
\hline
M$_0$ & 0.051 & 0.049 & 0.049 & 0.052 & 0.048 & 0.048 \\
M$_1$ & 0.059 & 0.051 & 0.047 & 0.058 & 0.050 & 0.050 \\
M$_2$ & 0.055 & 0.046 & 0.051 & 0.044 & 0.046 & 0.051 \\
M$_3$ & 0.054 & 0.054 & 0.047 & 0.059 & 0.051 & 0.048 \\
M$_4$ & 0.048 & 0.051 & 0.047 & 0.050 & 0.045 & 0.050 \\
M$_5$ & 0.049 & 0.047 & 0.058$^{(*)}$ & 0.046 & 0.049 & 0.054$^{(*)}$ \\
\hline
\end{tabular}\medskip
\caption{Empirical level of the test under $\cH_0$, for the configurations described above. We used $n \in \{ 50, 100, 500 \}$ and $1000$ replications. $^{(*)}$Simulations that needed more than one numerical trial, due to the explosive nature of the process and the large value of $n$.}
\label{TabH0}
\end{table}

For model M$_4$, it is important to note that Proposition \ref{PropGof} may not hold. Nevertheless, we know by virtue of Proposition \ref{PropUnivUnstable} that the statistic has the same order of magnitude, thus it seemed interesting to look at its empirical behavior in comparison with the other models (and we observe that it reacts as well). Now we turn to the frequency of rejection of $\cH_0$ when $f_0$ is not the true distribution, for the configuration $n=100$ and the $\cN(0,1)$ kernel. First of all, we are going to discuss on the hyperparameter $\delta$ that has to be non-arbitrary. A natural choice relies on the quantiles of the distribution that we test, to guarantee that a sufficient mass is covered by the statistic. As the simulations below will emphasize, at least 95\% seems relevant and accordingly, we let
\begin{equation*}
\delta = \max(-q_{0}(0.025), q_{0}(0.975))
\end{equation*}
where $q_{0}$ is the quantile function of $f_0$, which takes the simplified form $\delta = q_{0}(0.975)$ whenever the distribution is both centered and symmetrical. The latter situation is generally preferred since we deal with the noise of a regression model, but for example in our simulations we also need the general one for $\cN(m,1)$. Among the alternatives $\cN(m,1)$ and $\cN(0,\sigma^2)$, we pick two control models that will be $\cN(0.5,1)$ and $\cN(0,0.3)$, chosen to be the closest distributions for which our procedure is able to reach a maximal frequency of rejection under this configuration (see Figures \ref{FigH1moy} and \ref{FigH1var} below). Figure \ref{FigDel} is an illustration of the percentage of rejection (together with its 95\% confidence interval) of $\cH_0$ while we make $\delta$ increase, for the neutral model M$_0$. The corresponding values $\delta \approx 2.46$ and $\delta \approx 1.07$ that we suggest to use respectively for these examples seem therefore suitable.

\begin{figure}[h!]
\centering
\includegraphics[width=7.5cm]{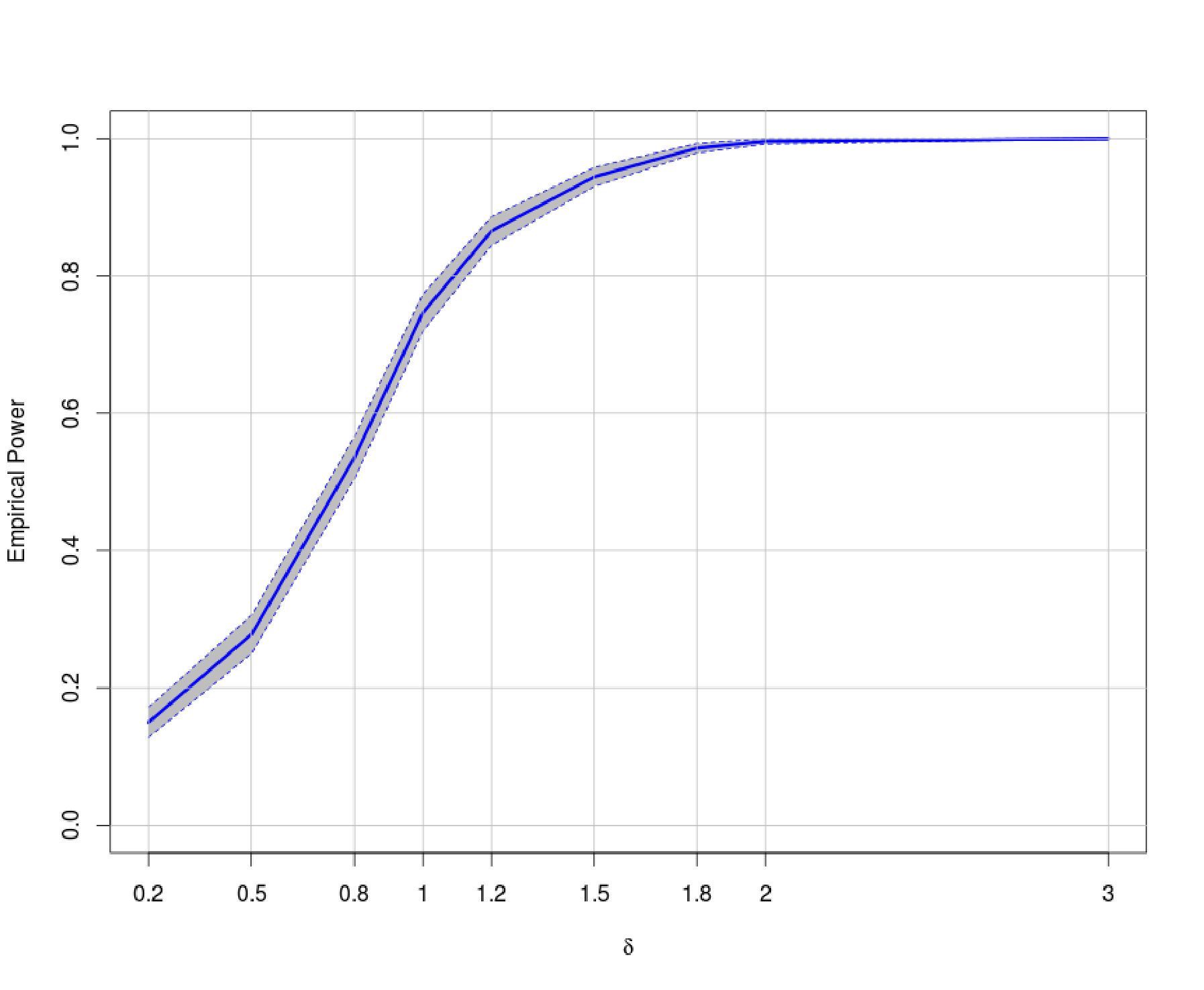}~\includegraphics[width=7.5cm]{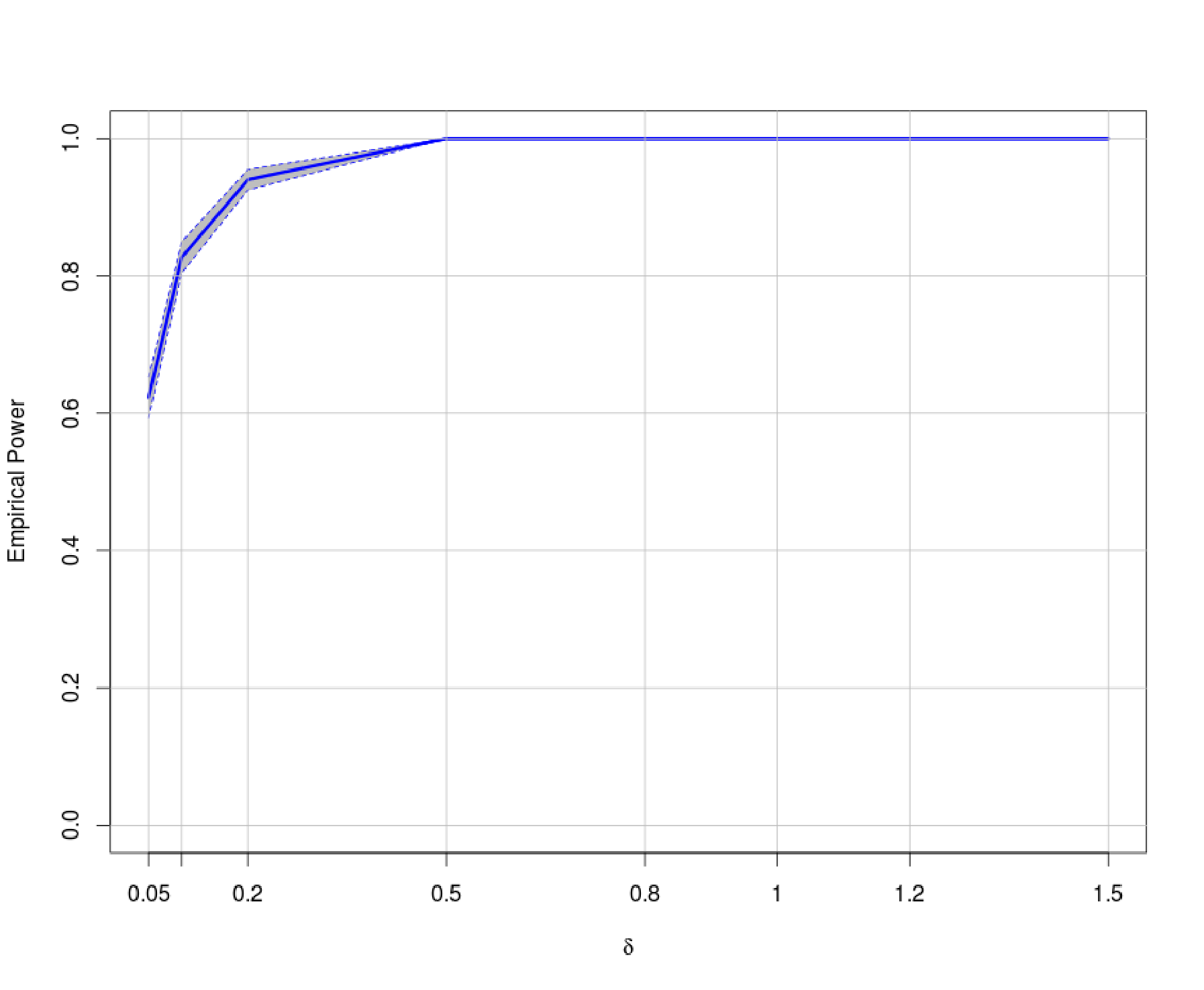}\medskip
\caption{Frequency of rejection of the test for $n=100$, $\dK=\cN(0,1)$, $h_0=0.14$ and $1000$ replications, depending on $\delta$ on the abscissa. The darkened areas are the 95\% confidence intervals. The alternatives are $\cN(0.5, 1)$ on the left, and $\cN(0, 0.3)$ on the right.}
\label{FigDel}
\end{figure}

Using this choice of $\delta$, we represent in Figures \ref{FigH1moy}--\ref{FigH1var} below the percentage of rejection of $\cH_0$ against the aforementioned $\cN(m,1)$ and $\cN(0,\sigma^2)$ alternatives for $f_0$, for different values of the parameters, to investigate the sensitivity towards location and scale. We also make experiments in Figure \ref{FigH1dist} with different distributions as alternatives (choosing $f_0$ as the Student, uniform, Laplace and Cauchy distributions, respectively). First of all, the main observation is that all models give very similar results (all curves are almost superimposed) even if $n$ is not so large. That corroborates the results of the paper: residuals from stable, explosive or some (univariate) unstable models satisfy the Bickel-Rosenblatt original convergence. Our procedure is roughly equivalent to the Kolmogorov-Smirnov one to test for location or scale in the Gaussian family (in fact it seems to be less powerful for location and slightly more powerful for scale). However, it appears that our procedure better managed to discriminate some alternatives with close but different distributions. The objective of the paper is mainly theoretical and of course, a much more extensive study is needed to give any permanent conclusion about the comparison (sensitivity in $n$, $h_0$, $\kappa$, $\dK$, $\delta$, role of the true distribution of $(\veps_{t})$, etc.). Also, we must not underestimate the difficulty to proceed to the numerical integrations required for the statistic, and the very likely improvements that may stem from in-depth applied studies.

\begin{figure}[h!]
\centering
\includegraphics[width=5cm]{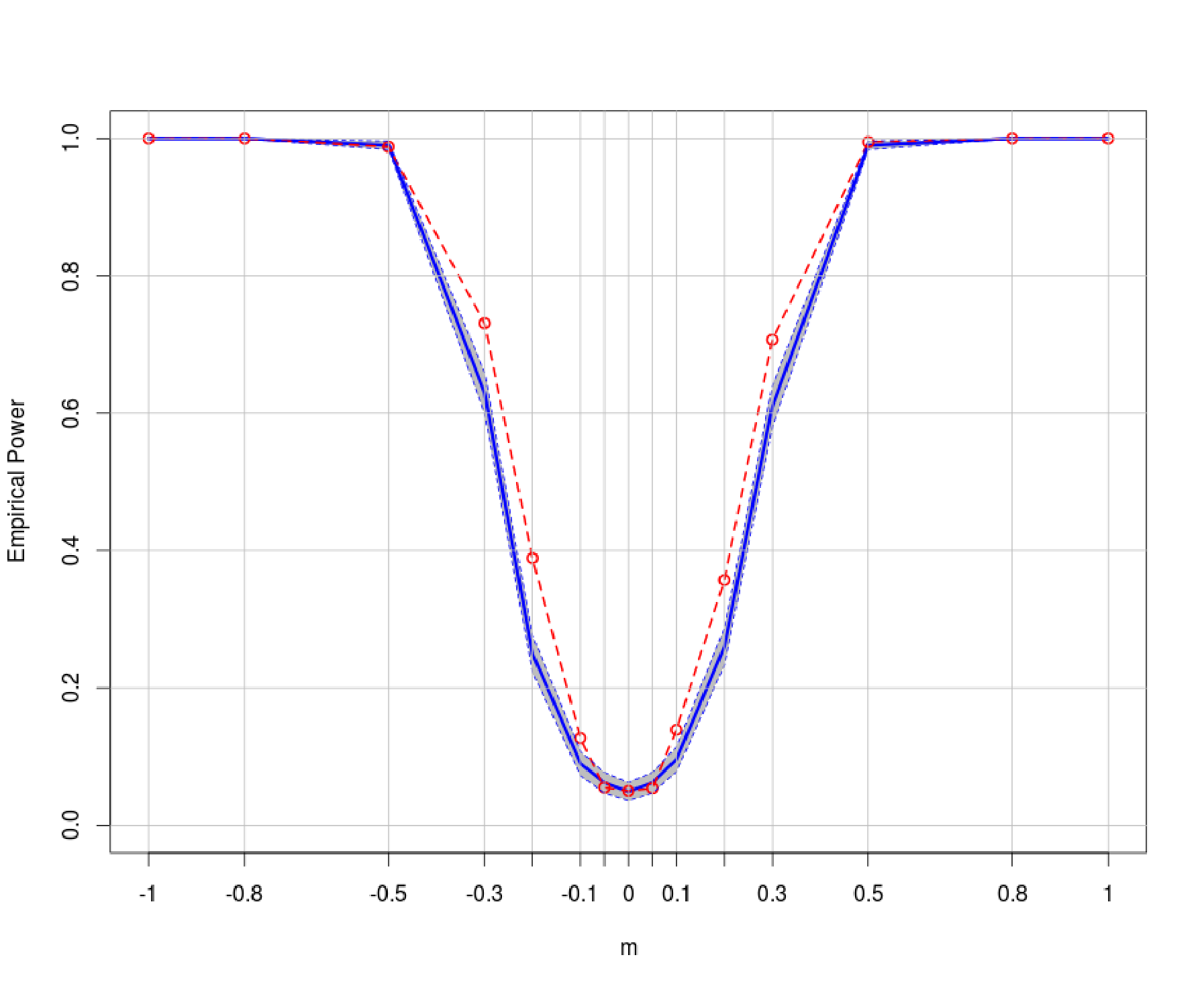}~\includegraphics[width=5cm]{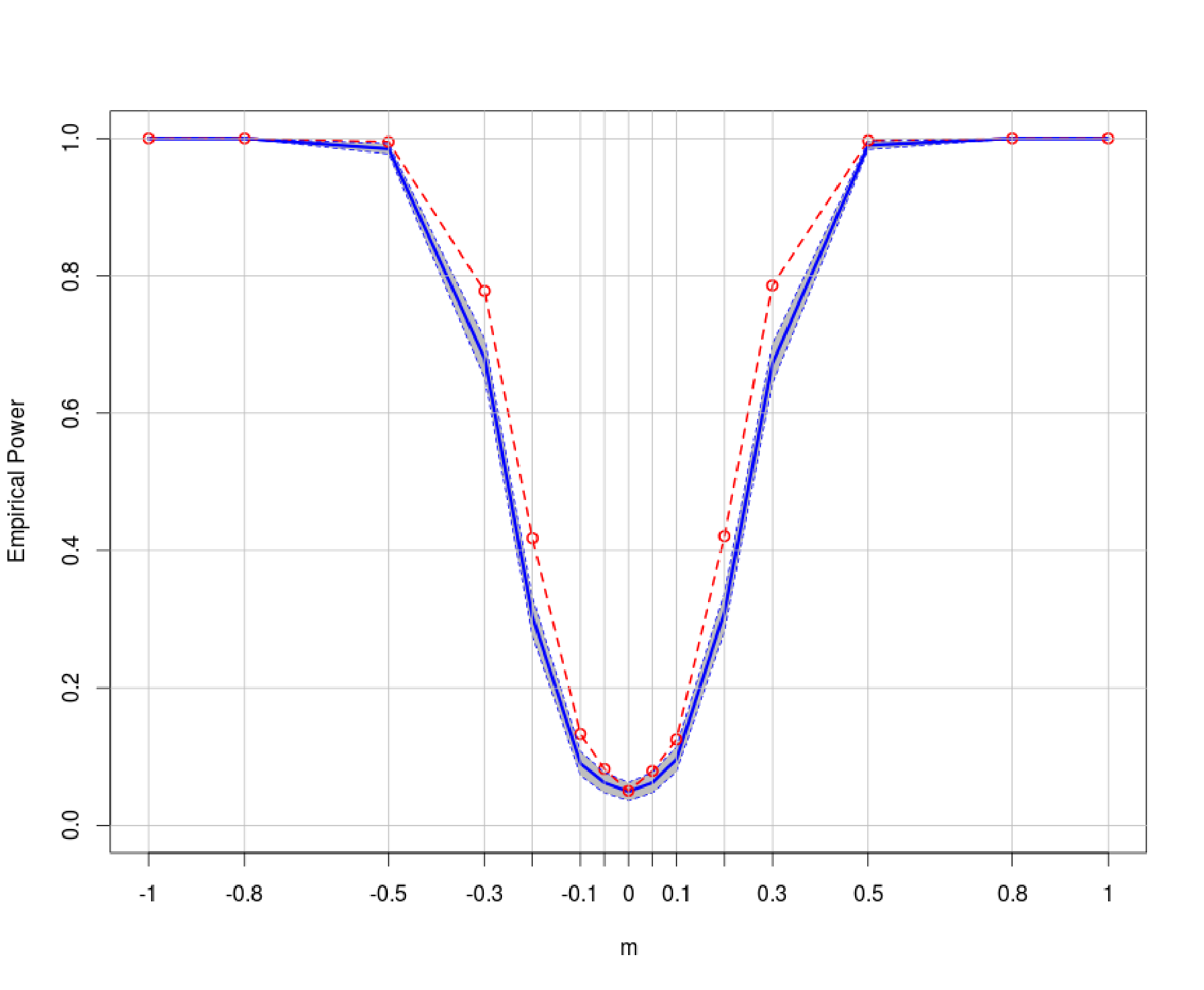}~\includegraphics[width=5cm]{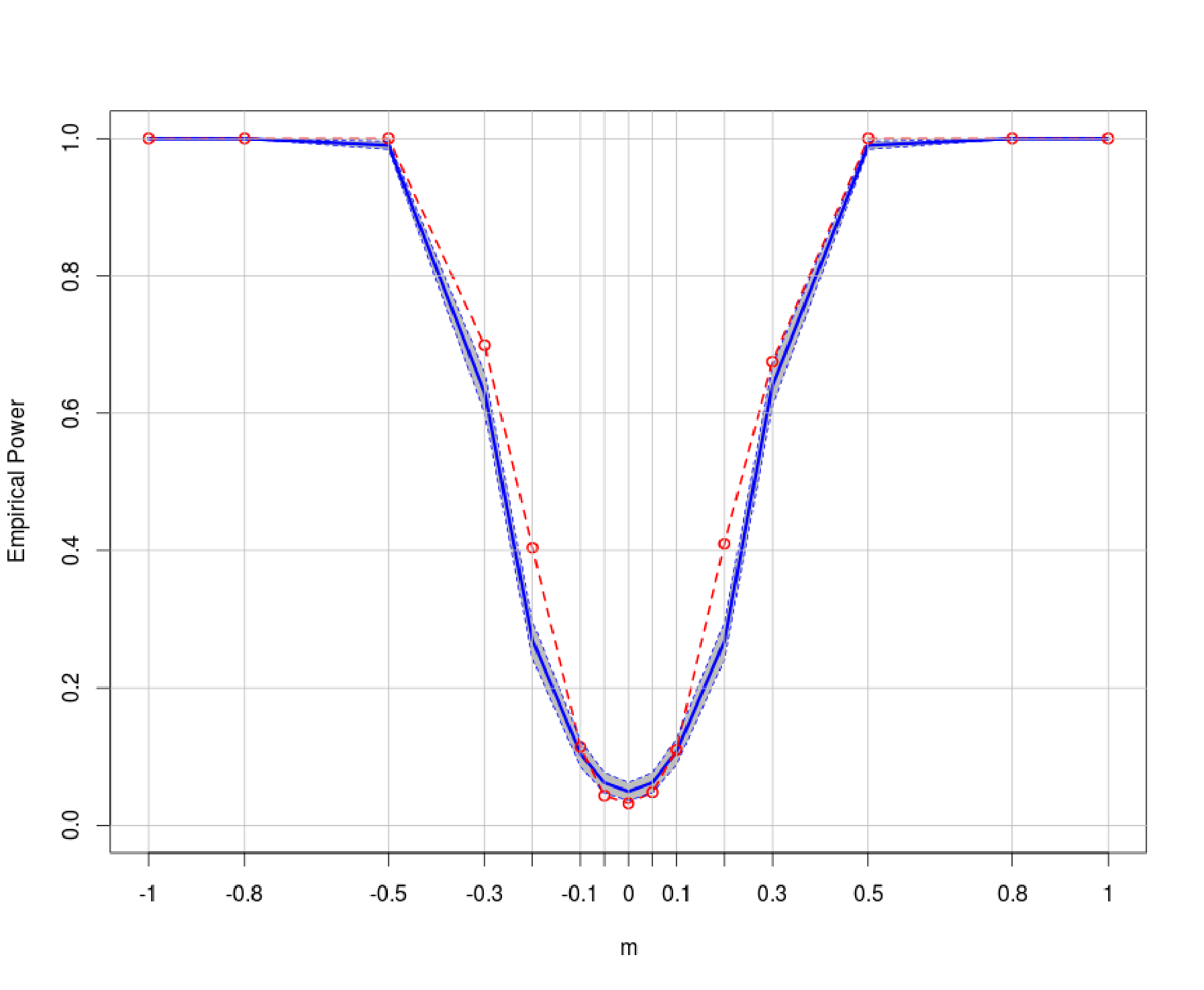}\\ \includegraphics[width=5cm]{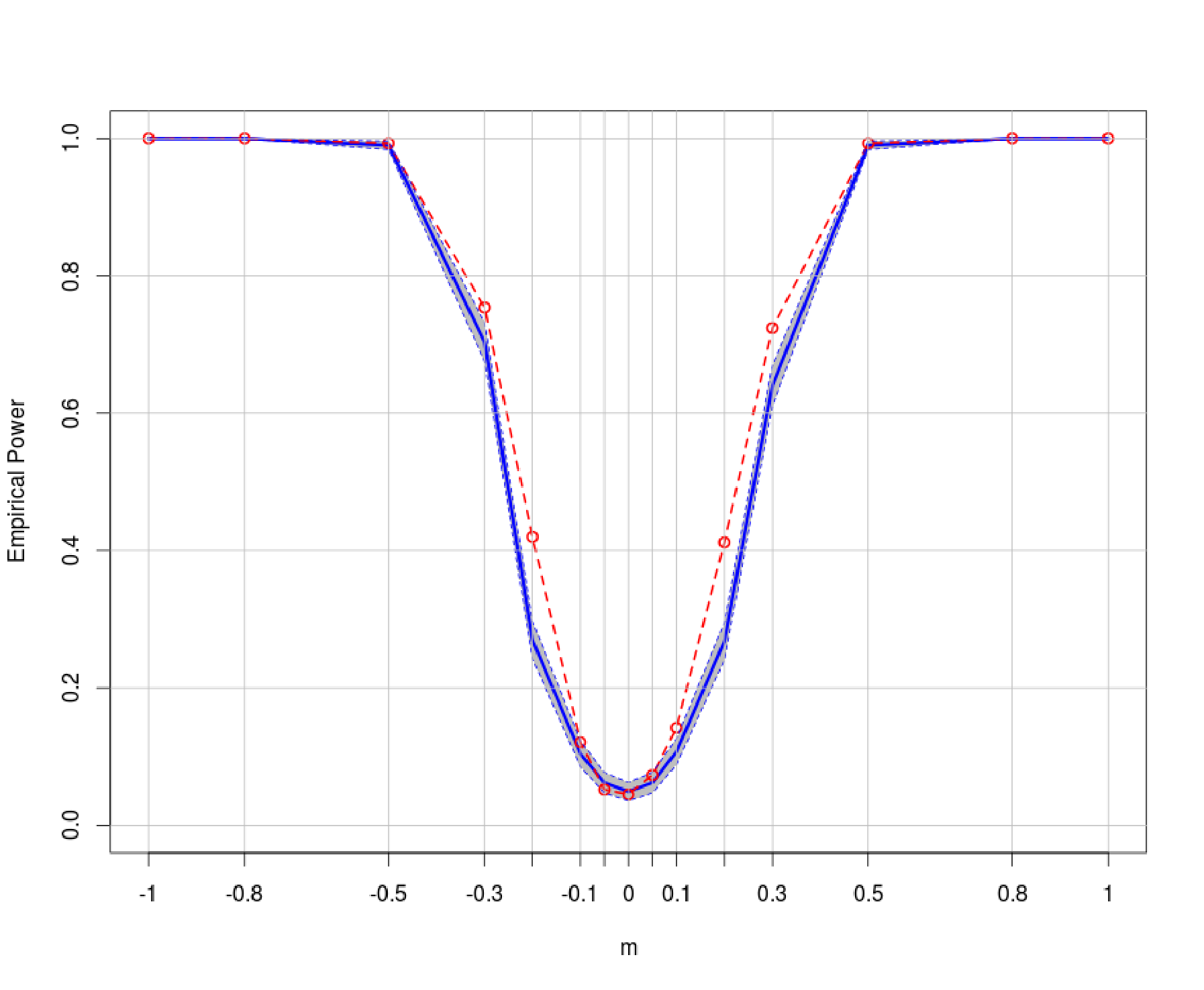}~\includegraphics[width=5cm]{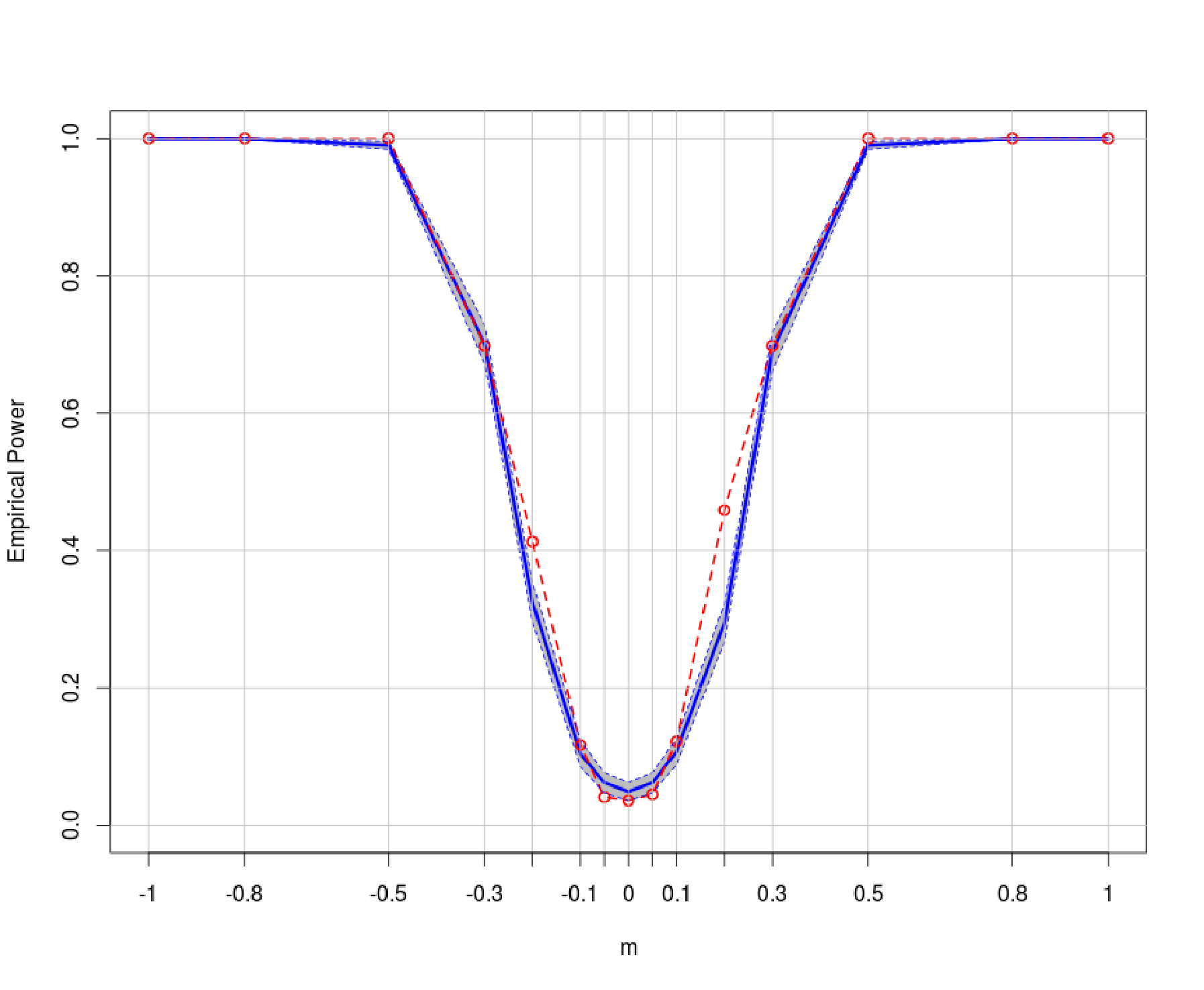}~\includegraphics[width=5cm]{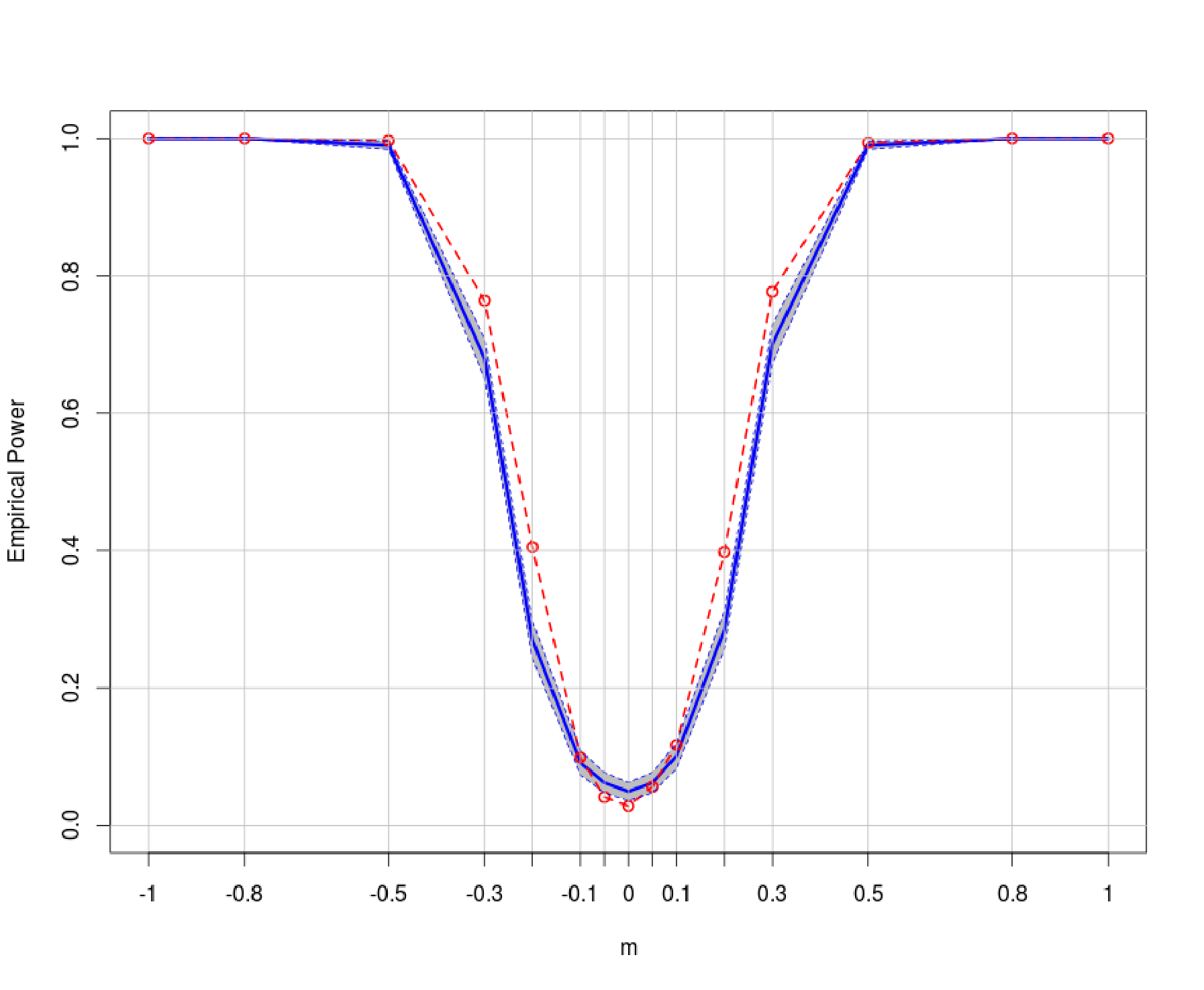}\medskip
\caption{Frequency of rejection of the test for $n=100$, $\dK=\cN(0,1)$, $h_0=0.14$ and $1000$ replications. The alternative $f_0$ are $\cN(m, 1)$ for some $m \in [ -1\,;\,1]$ with true value $m=0$. Results obtained from M$_0$, M$_1$ and M$_2$ are in the top (from left to right), results obtained from M$_3$, M$_4$ and M$_5$ are in the bottom (from left to right). The dotted line in red corresponds to the Kolmogorov-Smirnov test and the darkened areas are the 95\% confidence intervals.}
\label{FigH1moy}
\end{figure}

\begin{figure}[h!]
\centering
\includegraphics[width=5cm]{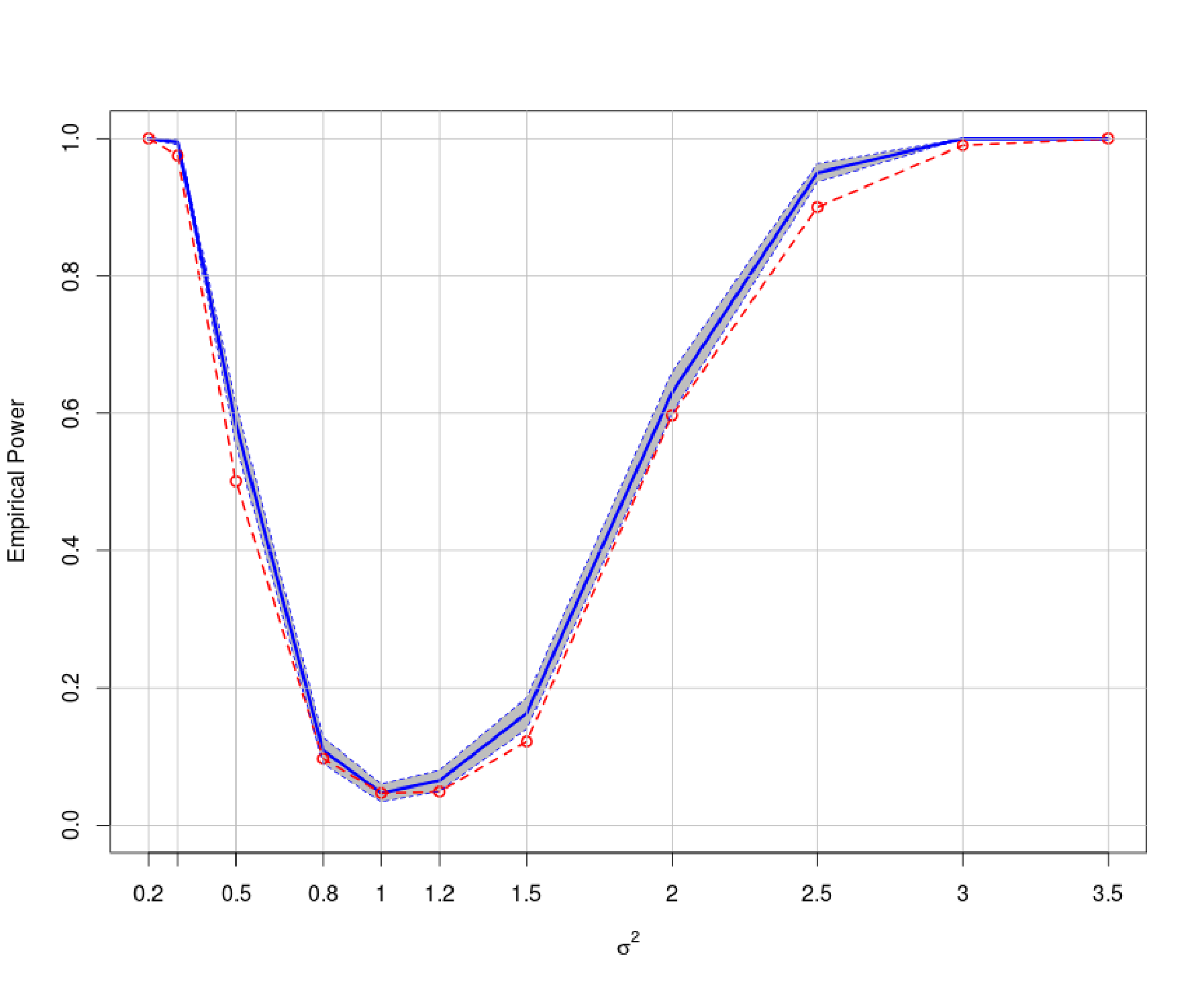}~\includegraphics[width=5cm]{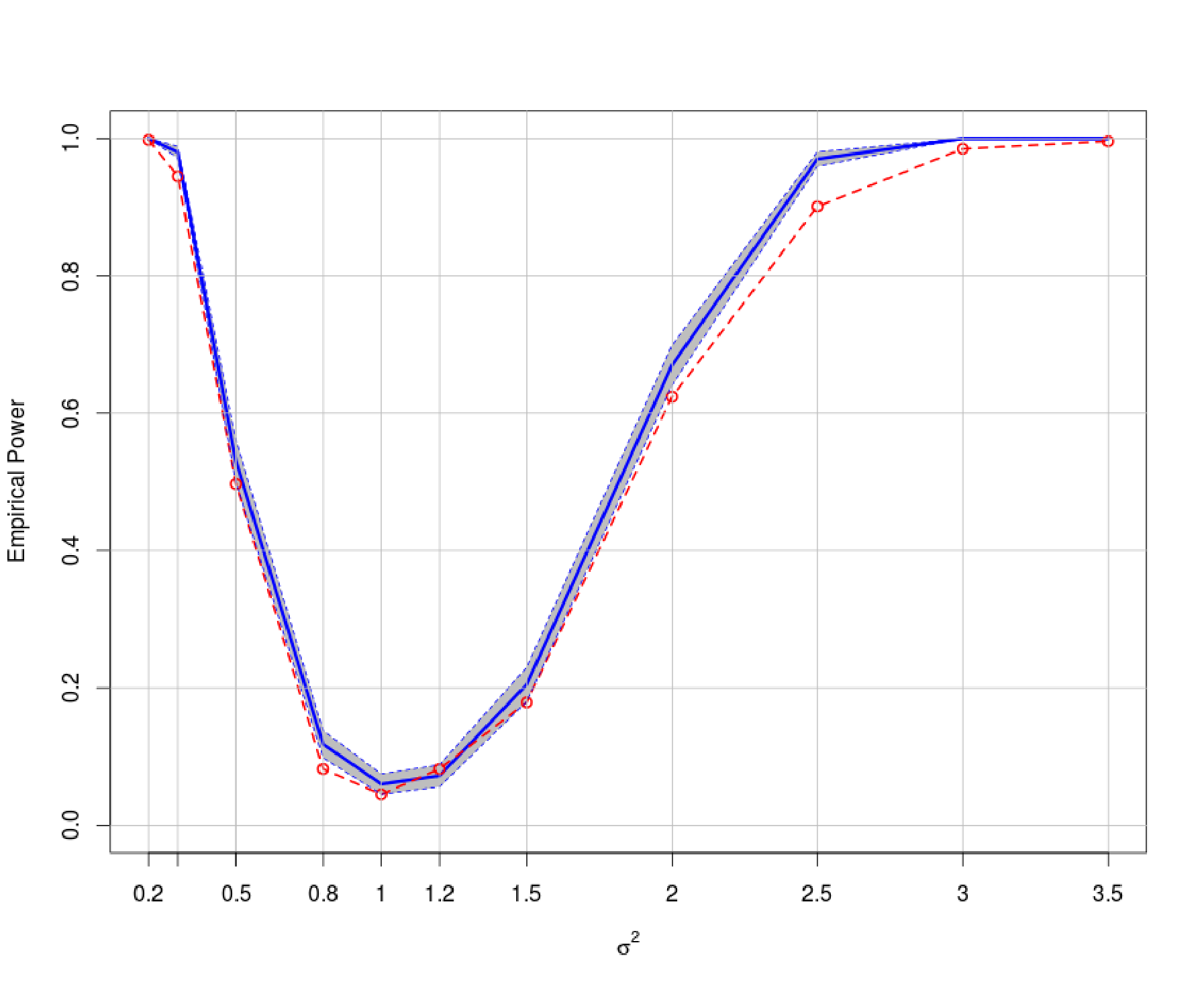}~\includegraphics[width=5cm]{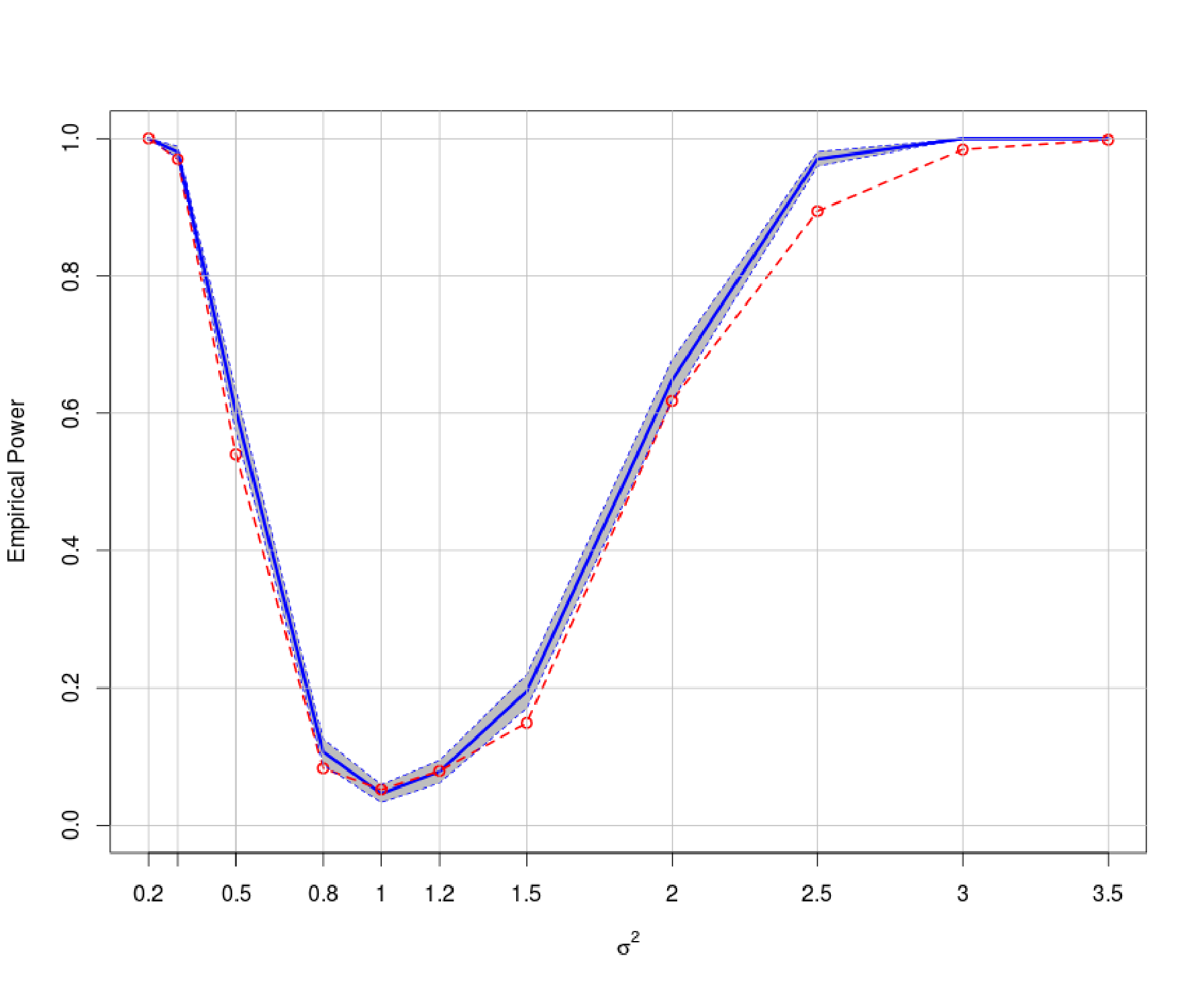}\\\includegraphics[width=5cm]{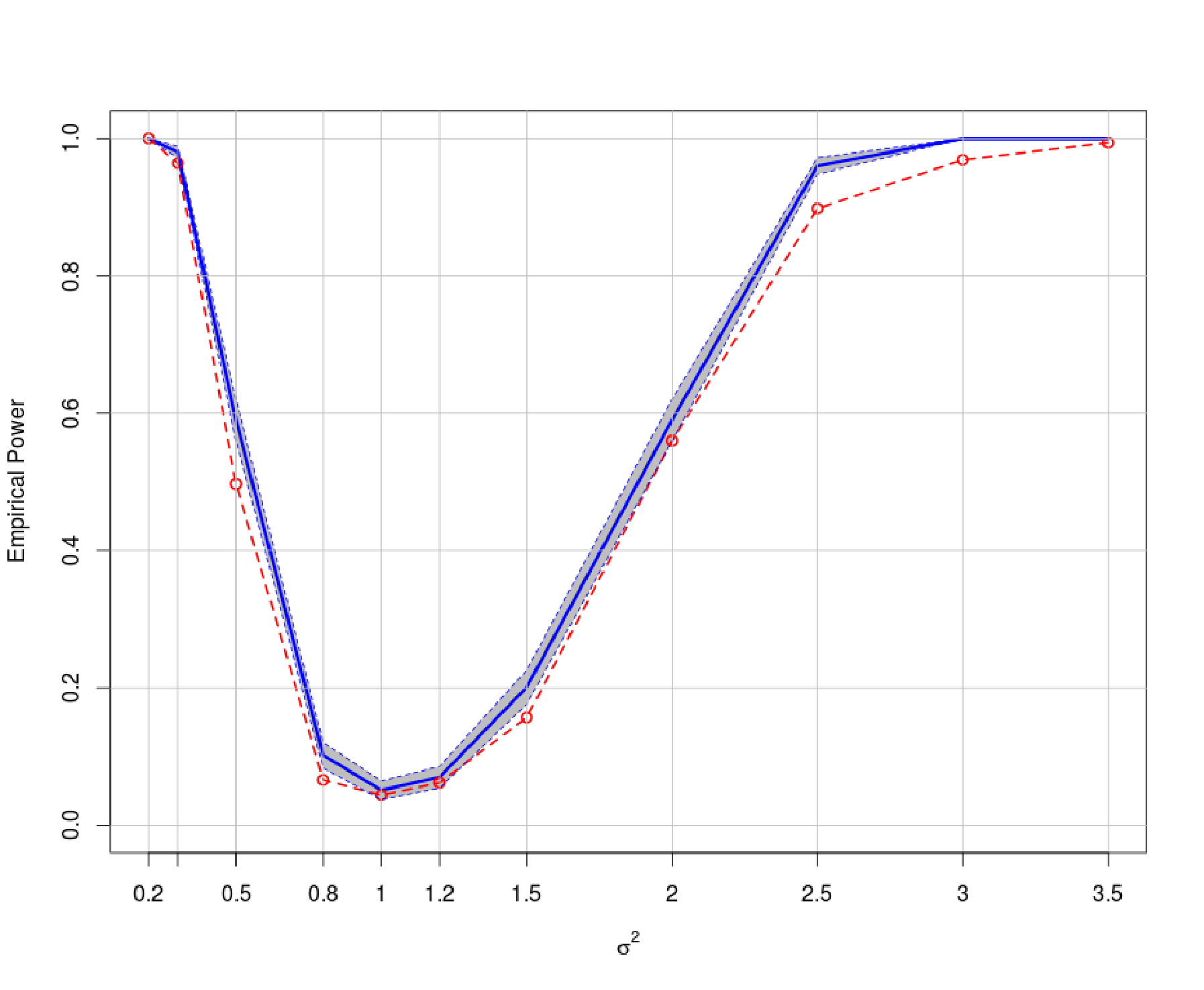}~\includegraphics[width=5cm]{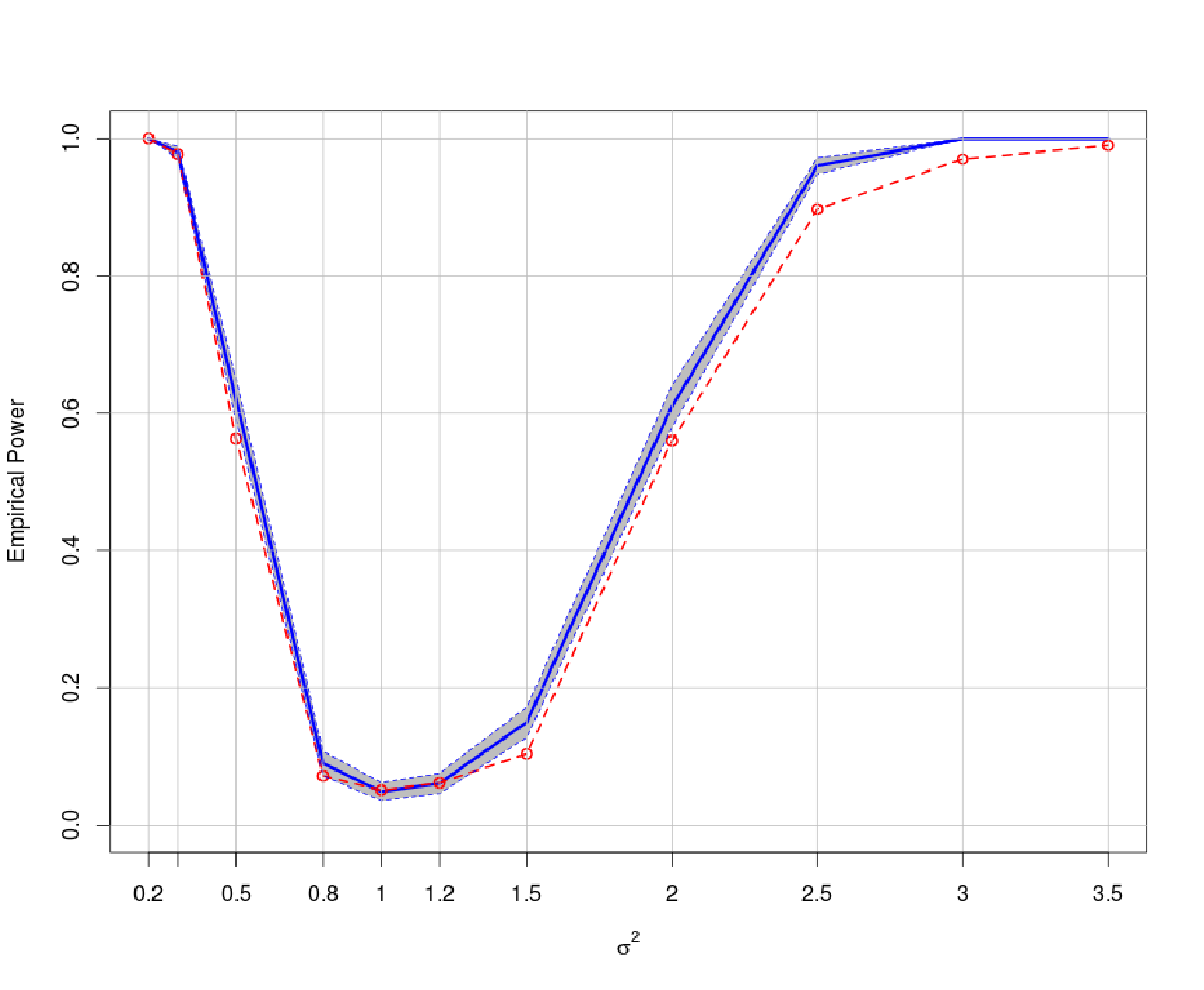}~\includegraphics[width=5cm]{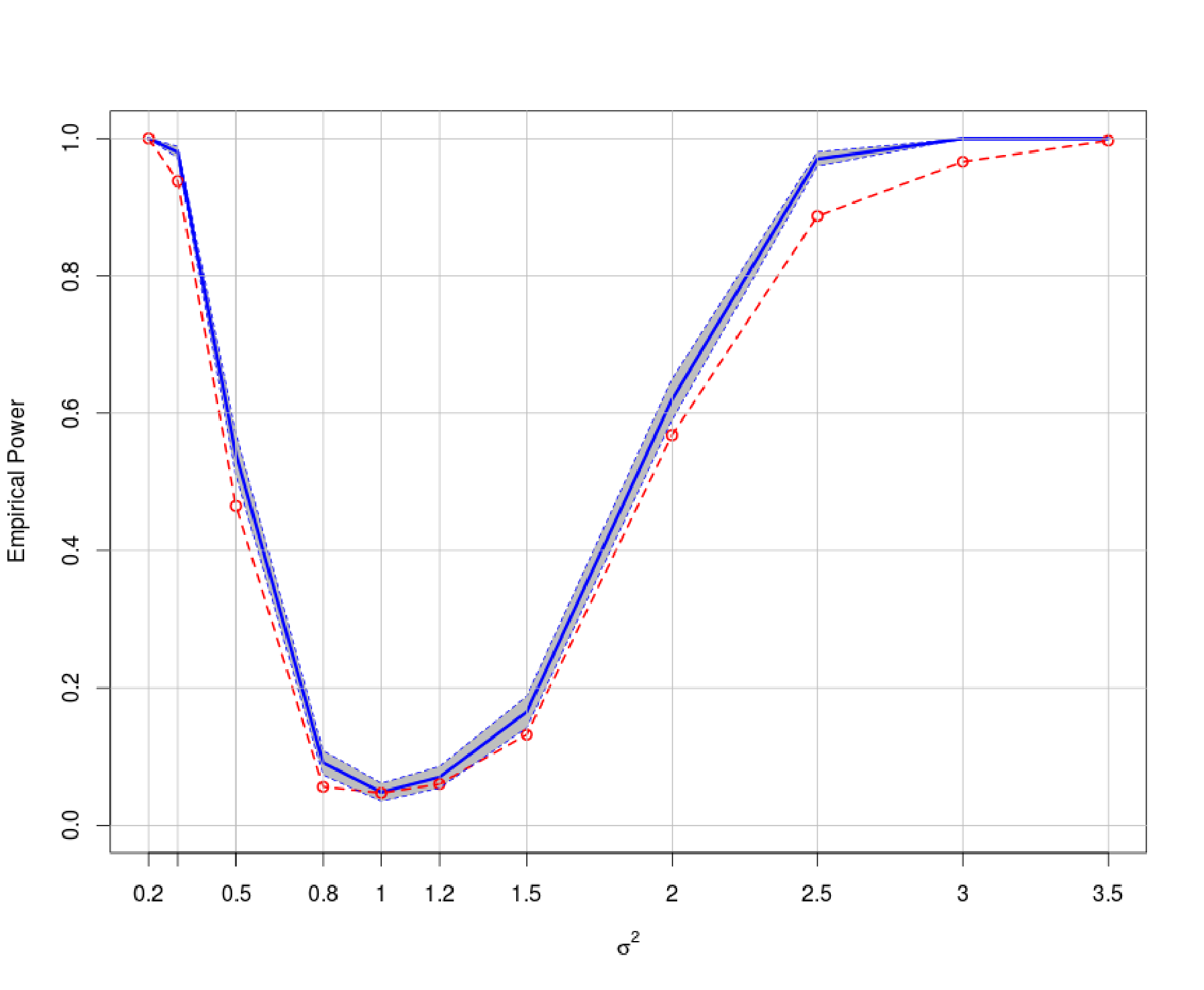}\medskip
\caption{Frequency of rejection of the test for $n=100$, $\dK=\cN(0,1)$, $h_0=0.14$ and $1000$ replications. The alternative $f_0$ are $\cN(0, \sigma^2)$ for some $\sigma^2 \in [ 0.2\,;\,3.5]$ with true value $\sigma^2=1$. Results obtained from M$_0$, M$_1$ and M$_2$ are in the top (from left to right), results obtained from M$_3$, M$_4$ and M$_5$ are in the bottom (from left to right). The dotted line in red corresponds to the Kolmogorov-Smirnov test and the darkened areas are the 95\% confidence intervals.}
\label{FigH1var}
\end{figure}

\begin{figure}[h!]
\centering
\includegraphics[width=12cm]{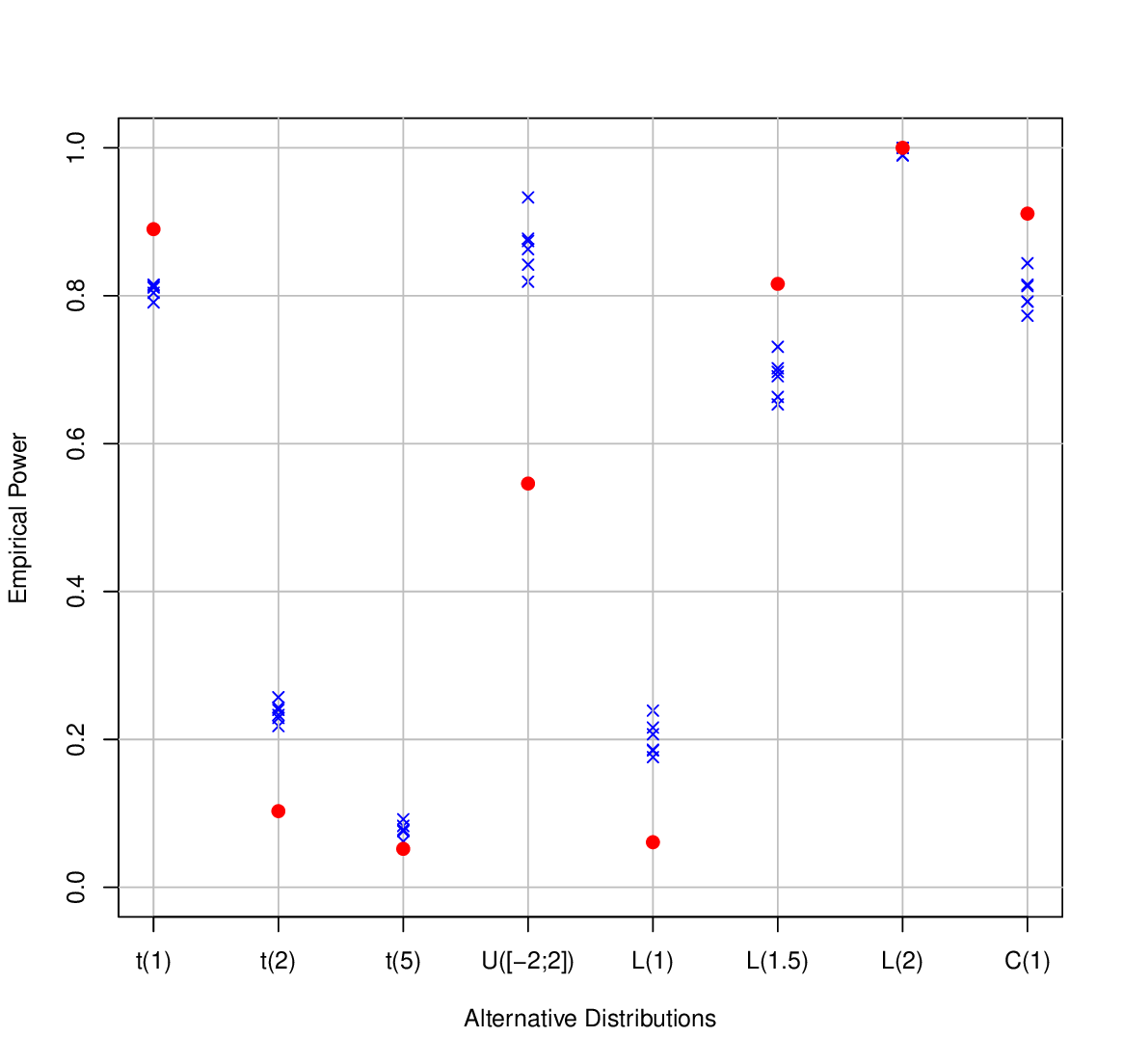}~\medskip
\caption{Frequency of rejection of the test for $n=100$, $\dK=\cN(0,1)$, $h_0=0.14$ and $1000$ replications. The alternative $f_0$ are from left to right the Student $t(1)$, $t(2)$, $t(5)$ distributions, the uniform $\cU([-2\,;\,2])$ distribution, the centered Laplace $\cL(1)$, $\cL(1.5)$ and $\cL(2)$ distributions and the centered Cauchy $\cC(1)$ distribution. Results obtained from M$_0$, $\hdots$, M$_5$ are superimposed as blue crosses. The red circles correspond to the Kolmogorov-Smirnov test for M$_0$.}
\label{FigH1dist}
\end{figure}

\begin{rem}
In the unstable case with $p=1$ and a positive unit root (namely, the random walk), and a kernel satisfying $\int_{\dR} \dK^{\prime}(s)\, \dd s \neq 0$, even if it is of lesser statistical interest, it is also possible to exploit Proposition \ref{PropUnivUnstable} to derive a statistical procedure. Indeed, let
\begin{equation*}
\sigma^2_0 = \int_{\dR} s^2\, f_0(s)\, \dd s \hsp \text{and} \hsp F_0 = \int_{-\delta}^{\delta} f_0(s)\, \dd s
\end{equation*}
with an adjustment of $\sigma^2_0$ if $f_0$ is not centered. Then, we can choose
\begin{equation*}
\widetilde{Z}_{n}^{\, 0} = \frac{h_{n} \big( \widetilde{T}_{n}^{\, 0} - \mu \big) }{\sigma_0\, \sqrt{F_0} \int_{\dR} \dK^{\prime}(s)\, \dd s}
\end{equation*}
and compare it with the quantiles associated with the distribution of
\begin{equation*}
Z = \left( \frac{\frac{1}{2}\, (W^2(1)-1)}{\int_0^1 W^2(u)\, \dd u}\, \int_0^1 W(u)\, \dd u \right)^{\! 2}.
\end{equation*}
\end{rem}

\noindent \textbf{Acknowledgments.} The authors thank Bernard Bercu for all his advices and suggestions during the preparation of this work. They also warmly thank the Associate Editor and the two anonymous Reviewers for their helpful and very constructive comments.

\section*{Appendix. Intermediate results}
\label{SecPrelim}

\setcounter{section}{1}
\setcounter{subsection}{0}
\setcounter{defi}{0}
\renewcommand{\thesection}{\Alph{section}}
\renewcommand{\thesubsection}{\Alph{section}.\arabic{subsection}}

We recall some (already known) preliminary results related to the behavior of $(X_{t})$, depending on the eigenvalues of $C_{\theta}$ and, in each case, the asymptotic behavior of the least-squares estimator. In the sequel, we assume that $\Phi_0$ shares the same assumptions of moments as $(\veps_{t})$. We only focus on results that will be useful in our proofs.

\subsection{Asymptotic behavior of the process}
\label{SecPrelimProcess_1}

\begin{prop}[Stable case]
\label{PropStable}
Assume that $(\veps_{t})$ is a strong white noise such that, for some $\nu \geq 1$, $\dE[ \vert \veps_1 \vert^{\nu}]$ is finite. If $(X_{t})$ satisfies \eqref{AR} and is stable (that is $\vert \lambda_1 \vert < 1$ or, equivalently, $\Theta(z) \neq 0$ for all $\vert z \vert \leq 1$), then as $n$ tends to infinity,
\begin{equation*}
\sum_{t=1}^{n} \vert X_{t} \vert^{\nu} = O(n) \as \hsp \text{and} \hsp \sup_{1\, \leq\, t\, \leq\, n} \vert X_{t} \vert = o(n^{1/\nu}) \as
\end{equation*}
In addition if $\nu \geq 2$, then
\begin{equation*}
\sum_{t=1}^{n} \Phi_{t} = O_{\P}(\sqrt{n}).
\end{equation*}
\end{prop}
\begin{proof}
See \cite[Lem. A.2]{BercuProia13} and \cite[Lem. B.2]{Proia13}.
The last result is deduced from the invariance principle of \cite[Thm. 1]{DedeckerRio00}, since $(X_{t})$ is (asymptotically) stationary in this case. Indeed, 
we clearly have
\begin{equation*}
\frac{1}{\sqrt{n}}\, \sum_{t=1}^{n} \Phi_{t} = \frac{1}{\sqrt{n}}\, \sum_{t=1}^{n} \Phi_{t}^{*} +  \frac{1}{\sqrt{n}}\, \bigg( \sum_{t=1}^{n} C_{\theta}^{\, t} \bigg) ( \Phi_0 - \Phi_0^{*} )
\end{equation*}
where $(\Phi_{t}^{*})$ is the (second-order) stationary version of the process for $\nu \geq 2$. We conclude that $(\Phi_{t})$ and $(\Phi_{t}^{*})$ share the same invariance principle (with rate $\sqrt{n}$), since $\rho(C_{\theta}) < 1$.
\end{proof}

\begin{prop}[Purely explosive case] 
\label{PropExplo}
\cite[Thm. 2]{LaiWei83} Assume that $(\veps_t)$ is a strong white noise having a finite variance $\sigma^2$. If $(X_t)$ satisfies \eqref{AR}--\eqref{VAR} and is purely explosive (that is $\vert \lambda_{p} \vert > 1$ or, equivalently, $\Theta(z) \neq 0$ for all $\vert z \vert \geq 1$), then
\begin{equation*}
\lim_{n \rightarrow +\infty} C_{\theta}^{-n}\, \Phi_{n} = \Phi_0 + \sum_{k=1}^{\infty} C_{\theta}^{-k}\, E_{k} = Z \as
\end{equation*}
and
\begin{equation*}
\lim_{n \rightarrow +\infty} C_{\theta}^{-n} \left(\sum_{t=1}^n \Phi_{t-1} \Phi_{t-1}\tr \right) (C_{\theta}^{-n})\tr = \sum_{k=1}^{\infty} C_{\theta}^{-k}\, Z\: Z\tr (C_{\theta}^{-k})\tr = G \as
\end{equation*}
In addition, $G$ is (a.s.) positive definite.
\end{prop}

\begin{prop}[Purely unstable case]
\label{PropUnstable}
Assume that $(\veps_t)$ is a strong white noise having a finite variance $\sigma^2$. If $(X_t)$ satisfies \eqref{AR}--\eqref{VAR} with $p=1$ and is unstable (that is $\vert \lambda \vert = 1$ or, equivalently, $\Theta(z) \neq 0$ for all $\vert z \vert \neq 1$), then for $k \in \dN$,
\begin{equation}\label{eq:purely1}
\sum_{t=1}^{n} X_{t}^{2\, k} = O_{\P}(n^{k+1}) \hsp \text{and} \hsp \vert X_{n} \vert = O_{\P}(n^{1/2}).
\end{equation}
In addition,
\begin{equation}\label{eq:purely2}
\sum_{t=1}^{n} X_{t} = \left\{
\begin{array}{ll}
O_{\P}(n^{3/2}) & \mbox{for } \lambda=1 \\
O_{\P}(n^{1/2}) & \mbox{for } \lambda=-1.
\end{array}
\right.
\end{equation}
If $(X_t)$ satisfies \eqref{AR}--\eqref{VAR} with $p=s$ and is seasonal unstable (that is $\theta_1 = \hdots = \theta_{s-1} = 0$ and $\theta_{s} = 1$), then \eqref{eq:purely1} and \eqref{eq:purely2} hold with $\lambda=\theta_s=1$.
\end{prop}

These results are detailed in Lemma \ref{LemUnstable1} and then proved.

\subsection{Asymptotic behavior of the least-squares estimator}
\label{SecPrelimProcess_2}

In this subsection, $(\veps_{t})$ is supposed to have a finite moment of order $2+\gamma$, for some $\gamma > 0$. The consistency results have been established in \cite[Thm. 1]{LaiWei83}, while the weak convergences are proved in \cite{BrockwellDavis06}, \cite[Thm. 2]{Stigum74} and \cite[Thm. 3.5.1]{ChanWei88} respectively.

\begin{prop}[Stable case]
\label{PropOLSStable}
In the stable case ($\vert \lambda_1 \vert < 1$), the least-squares estimator $\hTTn$ of $\theta$ is strongly consistent. In addition,
\begin{equation*}
\sqrt n\, (\hTTn - \theta) \cvgl \cN(0, \sigma^2\, \Gamma_{p}^{-1})
\end{equation*}
where $\Gamma_{p}$ is the $p \times p$ asymptotic covariance matrix of $(X_{t})$.
\end{prop}

\begin{prop}[Purely explosive case]
\label{PropOLSExplo}
In the purely explosive case ($\vert \lambda_{p} \vert > 1$), the least-squares estimator $\hTTn$ of $\theta$ is strongly consistent. In addition, 
\begin{equation*}
C_{\theta}^{n}\, (\hTTn - \theta) \cvgl U
\end{equation*}
where $U$ is a nondegenerate random vector given in (1.8)--(1.9)--(1.10) of \cite{Stigum74}.
\end{prop}

\begin{prop}[Purely unstable case]
\label{PropOLSUnstable}
In the purely unstable case ($\vert \lambda_1 \vert = \vert \lambda_{p} \vert = 1$), the least-squares estimator $\hTTn$ of $\theta$ is strongly consistent. In the univariate case ($p=1$), we have in addition
\begin{equation*}
n\, (\hTTn - \theta) \cvgl \textnormal{sgn}(\theta)\, \frac{\frac{1}{2}\, (W^2(1)-1)}{\int_0^1 W^2(u)\, \dd u}
\end{equation*}
where $(W(t),\, t \in [0,1])$ is a standard Wiener process and $\textnormal{sgn}(\theta)$ stands for the sign of $\theta$. 
In the seasonal case ($p = s$), since $\Theta(z) = 0$ is an equation admitting the complex $s$--th roots of unity as solutions,
\begin{equation*}
n\, (\hTTn - \theta) \cvgl S(W_{s})
\end{equation*}
where $S(W_{s})$ is a functional of a standard Wiener process $(W_{s}(t),\, t \in [0,1])$ of dimension $s$ that can be explicitly determined following \cite[Thm. 3.5.1]{ChanWei88}.
\end{prop}

\section*{Appendix. Proofs of the main results}
\label{SecProof}

\setcounter{section}{2}
\setcounter{subsection}{0}

In this section, we prove our results.

\subsection{Stable case - Proof of Theorem \ref{ThmStable}}
\label{SecProofStable}

We consider the case where the kernel satisfies hypothesis  \ref{hyp:H1}. The difference between the statistics can be expressed like
\begin{equation}
\label{DecompStable}
\widehat{T}_{n} - T_{n} = n\, h_{n}\, \int_{\dR} \big( \widehat{f}_{n}(x) - f_{n}(x) \big)^2 a(x)\, \dd x + R_{n}
\end{equation}
with 
\begin{equation*}
R_{n} = 2\, n\, h_{n}\, \int_{\dR} \big( \widehat{f}_{n}(x) - f_{n}(x) \big)\, \big( f_{n}(x) - (\dK_{h_{n}}*f)(x) \big)\, a(x)\, \dd x
\end{equation*}
and $f_n$ is the analog of $\widehat f_n$ defined in \eqref{PREst} where the residuals $\widehat \varepsilon$ have been replaced by the strong white noise $\varepsilon$. Recall that $T_n$ has been defined before \eqref{AsNormBR}. Now we follow an idea illustrated in the proof of \cite[Thm. 1.1]{HorvathZitikis04} and in that of \cite[Thm. 2.1]{LeeNa02}, which consists in using Taylor expansion to get, for all $1 \leq t \leq n$,
\begin{equation*}
\dK\left( \frac{x-\widehat{\veps}_{t}}{h_{n}} \right) - \dK\left( \frac{x-\veps_{t}}{h_{n}} \right) = \frac{\veps_{t} - \widehat{\veps}_{t}}{h_{n}}\, \dK^{\prime}\left( \frac{x-\veps_{t}}{h_{n}} \right) + \frac{(\veps_{t} - \widehat{\veps}_{t})^2}{2\, h_{n}^2}\, \dK^{\prime \prime}\left( \Delta_{t x}\right)
\end{equation*}
where
\begin{equation*}
\Delta_{t x} = \frac{ x - \veps_{t} + \zeta (\veps_{t} - \widehat{\veps}_{t})}{h_{n}}
\end{equation*}
for some $0 < \zeta < 1$. Note that 
\begin{equation*}
\veps_{t} - \widehat{\veps}_{t} = (\hTTn - \theta)\tr \Phi_{t-1} = \langle \hTTn - \theta, \Phi_{t-1} \rangle
\end{equation*}
and obviously that
\begin{equation*}
\big\vert \langle \hTTn - \theta, \Phi_{t-1} \rangle \big\vert \leq \Vert \hTTn - \theta \Vert\, \Vert \Phi_{t-1} \Vert.
\end{equation*}
On the one hand, we consider the first term (say, $I_{n}$) of the right-hand side of \eqref{DecompStable} that can be bounded like
\begin{eqnarray}
I_{n} & = & n\, h_{n}\, \int_{\dR} \left\{ \frac{1}{n\, h_{n}}\, \sum_{t=1}^{n} \left( \dK\left( \frac{x-\widehat{\veps}_{t}}{h_{n}} \right) - \dK\left( \frac{x-\veps_{t}}{h_{n}} \right) \right) \right\}^{\! 2} a(x)\, \dd x \nonumber \\
 & \leq & \frac{2}{n\, h_{n}^3}\, \int_{\dR} \left\{ \sum_{t=1}^{n} \langle \hTTn - \theta, \Phi_{t-1} \rangle\, \dK^{\prime}\left( \frac{x-\veps_{t}}{h_{n}} \right) \right\}^{\! 2} a(x)\, \dd x \nonumber \\
 & & \hsp \hsp + ~ \frac{1}{2\, n\, h_{n}^5}\, \int_{\dR} \left\{ \sum_{t=1}^{n}  \langle \hTTn - \theta, \Phi_{t-1} \rangle ^2\, \dK^{\prime \prime}\left( \Delta_{t x} \right) \right\}^{\! 2} a(x)\, \dd x. \label{DecompStable2T}
\end{eqnarray}
At this step, we need two technical lemmas.
\begin{lem}
\label{LemStable1}
We have
\begin{equation*}
I_{1,n} = \int_{\dR} \left\{ \sum_{t=1}^{n} \langle \hTTn - \theta, \Phi_{t-1} \rangle\, \dK^{\prime}\left( \frac{x-\veps_{t}}{h_{n}} \right) \right\}^{\! 2} a(x)\, \dd x = O_{\P}(h_{n}).
\end{equation*}
\end{lem}
\begin{proof}
One clearly has
\begin{eqnarray}
I_{1,n} & \leq & 2\, \Bigg\{ \int_{\dR} \left( \sum_{t=1}^{n} \langle \hTTn - \theta, \Phi_{t-1} \rangle\, v_{t}(x) \right)^{\! 2} a(x)\, \dd x \nonumber \\
 & & \hsp \hsp + ~ \int_{\dR} \left( \sum_{t=1}^{n} \langle \hTTn - \theta, \Phi_{t-1} \rangle\, e(x) \right)^{\! 2} a(x)\, \dd x \Bigg\} ~ = ~ 2\, ( J_{1,n} + J_{2,n} ) \label{DecompStableIn1}
\end{eqnarray}
where
\begin{equation*}
v_{t}(x) = \dK^{\prime}\left( \frac{x-\veps_{t}}{h_{n}} \right) - \dE\left[ \dK^{\prime}\left( \frac{x-\veps_1}{h_{n}} \right) \right] \hsp \text{and} \hsp e(x) = \dE\left[ \dK^{\prime}\left( \frac{x-\veps_1}{h_{n}} \right) \right].
\end{equation*}
Now let us consider $J_{1,n}$ and $J_{2,n}$. First,
\begin{eqnarray}
J_{1,n} & \leq & \Vert\, \hTTn - \theta \Vert^2 \int_{\dR} \bigg\Vert \sum_{t=1}^{n} \Phi_{t-1}\, v_{t}(x) \bigg\Vert^2 a(x)\, \dd x \nonumber \\
 & = & \Vert \hTTn - \theta \Vert^2\, \sum_{i=1}^{p} \int_{\dR} \left( \sum_{t=1}^{n}  X_{t-i}\, v_{t}(x) \right)^{\! 2} a(x)\, \dd x. \label{DecompStableJ1}
\end{eqnarray}
Let $i \in \{1, \hdots, p\}$. Then,
\begin{eqnarray*}
\dE\left[ \int_{\dR} \left( \sum_{t=1}^{n}  X_{t-i}\, v_{t}(x) \right)^{\! 2} a(x)\, \dd x \right] & = & \int_{\dR} \dE\left[ \left( \sum_{t=1}^{n} X_{t-i}\, v_{t}(x) \right)^{\! 2} \right]\, a(x)\, \dd x \\
 & = & \int_{\dR} \sum_{t=1}^{n} \dE\big[ X_{t-i}^{\, 2} \big]\, \dE\big[ v_{t}^{\, 2}(x) \big] a(x)\, \dd x \\
 & = & \sum_{t=1}^{n} \dE\big[ X_{t-i}^{\, 2} \big]\, \int_{\dR} \dE\big[ v_{t}^{\, 2}(x) \big]\, a(x)\, \dd x.
\end{eqnarray*}
But, under our assumptions, we recall that $\dE[X_{n}^{\, 2}] = O(1)$ from the asymptotic stationarity of the process and
\begin{eqnarray*}
\int_{\dR} \dE\big[ v_{t}^{\, 2}(x) \big]\, a(x)\, \dd x & \leq & \int_{\dR} \dE\left[ \left( \dK^{\prime}\left( \frac{x-\veps_1}{h_{n}} \right) \right)^{\! 2} \right] a(x)\, \dd x \\
 & = & h_{n} \int_{\dR}  (\dK^{\prime}(z))^2\, \int_{\dR} f(x - h_{n}\, z)\, a(x)\, \dd x\, \dd z ~ = ~ O(h_{n})
\end{eqnarray*}
since $\Var(Z)\leq \E[Z^2]$ for some random variable $Z$ and by  \ref{hyp:H0} and  \ref{hyp:H1}. Thus $J_{1,n} = O_{\P}(h_{n})$ \textit{via} Proposition \ref{PropOLSStable}. Then, by a direct calculation,
\begin{eqnarray}
J_{2,n} & = & \left( (\hTTn - \theta )\tr\, \sum_{t=1}^{n} \Phi_{t-1} \right)^{\! 2} \int_{\dR} e^{\, 2}(x)\, a(x)\, \dd x \nonumber \\
 & = & h_{n}^2 \left( (\hTTn - \theta )\tr\, \sum_{t=1}^{n} \Phi_{t-1} \right)^{\! 2} \int_{\dR} \left( \int_{\dR} \dK^{\prime}(z)\, f(x - h_{n}\, z)\, \dd z \right)^{\! 2} a(x)\, \dd x. \label{DecompStableJ2}
\end{eqnarray}
Using Proposition \ref{PropStable}, it follows that $J_{2,n} = O_{\P}(h_{n}^2)$.
\end{proof}

\begin{lem}
\label{LemStable2}
We have
\begin{equation*}
I_{2,n} = \int_{\dR} \left\{ \sum_{t=1}^{n}  \langle \hTTn - \theta, \Phi_{t-1} \rangle ^2\, \dK^{\prime \prime}\left( \Delta_{t x} \right) \right\}^{\! 2} a(x)\, \dd x = O_{\P}\!\left( h_{n}^2 + \frac{1}{n\, h_{n}^2} \right).
\end{equation*}
\end{lem}
\begin{proof}
We directly get
\begin{eqnarray}
I_{2,n} & \leq & \Vert \hTTn - \theta \Vert^4\, \int_{\dR} \left\{ \sum_{t=1}^{n} \Vert \Phi_{t-1} \Vert^2\, \dK^{\prime \prime}\left( \Delta_{t x} \right) \right\}^{\! 2} a(x)\, \dd x \nonumber \\
 & \leq & 3\, \Vert \hTTn - \theta \Vert^4\, \Bigg( \int_{\dR} \left\{ \sum_{t=1}^{n} \Vert \Phi_{t-1} \Vert^2\, \left( \dK^{\prime \prime}\left( \Delta_{t x} \right) - \dK^{\prime \prime}\left( \frac{x - \veps_{t}}{h_{n}} \right) \right) \right\}^{\! 2} a(x)\, \dd x \nonumber \\
 & & \hsp \hsp + ~ \int_{\dR} \left\{ \sum_{t=1}^{n} \Vert \Phi_{t-1} \Vert^2\, \left( \dK^{\prime \prime}\left( \frac{x - \veps_{t}}{h_{n}} \right) - \dE\left[ \dK^{\prime \prime}\left( \frac{x - \veps_1}{h_{n}} \right) \right] \right) \right\}^{\! 2} a(x)\, \dd x \nonumber \\
 & & \hsp \hsp + ~ \int_{\dR} \left\{ \sum_{t=1}^{n} \Vert \Phi_{t-1} \Vert^2\, \dE\left[ \dK^{\prime \prime}\left( \frac{x - \veps_1}{h_{n}} \right) \right] \right\}^{\! 2} a(x)\, \dd x \Bigg) \nonumber \\
 & = & 3\, \Vert \hTTn - \theta \Vert^4\, ( K_{1,n} + K_{2,n} + K_{3,n} ). \label{DecompStableIn2}
\end{eqnarray}
From the mean value theorem, under our hypotheses,
\begin{eqnarray*}
\dK^{\prime \prime}\left( \Delta_{t x} \right) - \dK^{\prime \prime}\left( \frac{x - \veps_{t}}{h_{n}} \right) & = & \dK^{\prime \prime}\left( \frac{ x - \veps_{t} + \zeta (\veps_{t} - \widehat{\veps}_{t})}{h_{n}} \right) - \dK^{\prime \prime}\left( \frac{x - \veps_{t}}{h_{n}} \right) \\
 & = & \frac{\zeta(\veps_{t} - \widehat{\veps}_{t})}{h_{n}}\, \dK^{\prime \prime \prime}\!\left(\frac{x - \veps_{t} + \xi\, \zeta (\veps_{t} - \widehat{\veps}_{t})}{h_{n}} \right)
\end{eqnarray*}
for some $0 < \zeta, \xi < 1$. We deduce that
\begin{equation*}
\left\vert\dK^{\prime \prime}\left( \Delta_{t x} \right) - \dK^{\prime \prime}\left( \frac{x - \veps_{t}}{h_{n}} \right) \right\vert \leq \frac{\vert \veps_{t}-\widehat{\veps}_{t} \vert}{h_{n}}\, \left\vert \dK^{\prime \prime \prime}\!\left(\frac{x - \veps_{t} + \xi\, \zeta (\veps_{t} - \widehat{\veps}_{t})}{h_{n}} \right) \right\vert.
\end{equation*}
Consequently, since $\dK^{\prime \prime \prime}$ is bounded,
\begin{eqnarray}
K_{1,n} ~ \leq ~ \frac{C}{h_{n}^2}\, \left( \sum_{t=1}^{n} \Vert \Phi_{t-1} \Vert^2\, \vert \veps_{t} - \widehat{\veps}_{t} \vert \right)^{\! 2} & \leq & \frac{C}{h_{n}^2}\, \sum_{t=1}^{n} \Vert \Phi_{t-1} \Vert^4\, \sum_{t=1}^{n} (\veps_{t} - \widehat{\veps}_{t} )^2 \nonumber \\
 & \leq & \frac{C}{h_{n}^2}\, \Vert \hTTn - \theta \Vert^2\, \sum_{t=1}^{n} \Vert \Phi_{t-1} \Vert^4\, \sum_{t=1}^{n} \Vert \Phi_{t-1} \Vert^2 \label{DecompStableK1}
\end{eqnarray}
for some constants. Then, $K_{1,n} = O_{\P}(n\, h_{n}^{-2})$ as soon as we suppose that $(\veps_{t})$ has a finite moment of order $\nu=4$ (by virtue of Proposition \ref{PropStable}). Now we proceed as for $J_{1,n}$ to get
\begin{eqnarray}
\dE[K_{2,n}] & \leq & \sum_{t=1}^{n} \dE\big[ \Vert \Phi_{t-1} \Vert^4 \big]\, \int_{\dR} \dE\left[  \dK^{\prime \prime} \left( \frac{x - \veps_1}{h_{n}} \right) ^{\! 2} \right] a(x)\, \dd x \nonumber \\
 & = & h_{n}\, \sum_{t=1}^{n} \dE\big[ \Vert \Phi_{t-1} \Vert^4 \big]\, \int_{\dR}  (\dK^{\prime \prime}(z))^2\, \int_{\dR} f(x - h_{n}\, z)\, a(x)\, \dd x\, \dd z ~ = ~ O(n\, h_{n}) \label{DecompStableK2}
\end{eqnarray}
which shows $K_{2,n} = O_{\P}(n\, h_{n})$, since $\dE\big[ X_{n}^{\, 4} \big] = O(1)$ under the hypotheses of stability and fourth-order moments. Finally,
\begin{eqnarray}
K_{3,n} & \leq & h_{n}^2\, \int_{\dR} \left( \int_{\dR} \vert \dK^{\prime \prime}(z) \vert\, f(x - h_{n}\, z)\, \dd z \right)^2 a(x)\, \dd x\, \left( \sum_{t=1}^{n} \Vert \Phi_{t-1} \Vert^2 \right)^{\! 2} \nonumber \\
 & \leq & C\, h_{n}^2\, \left( \sum_{t=1}^{n} \Vert \Phi_{t-1} \Vert^2 \right)^{\! 2} ~ = ~ O_{\P}(n^2\, h_{n}^2) \label{DecompStableK3}
\end{eqnarray}
for a constant $C$. Whence we deduce that $I_{2,n} = O_{\P}(h_{n}^2 + n^{-1} h_{n}^{-2})$.
\end{proof}
We are now ready to conclude the proof of Theorem \ref{ThmStable}. If we come back to $I_{n}$ in \eqref{DecompStable2T}, then the combination of Lemmas \ref{LemStable1} and \ref{LemStable2} leads to 
\begin{equation}
\label{CvgInStable}
\frac{I_{n}}{\sqrt{h_{n}}} = O_{\P}\!\left( \frac{1}{n\, h_{n}^{5/2}} + \frac{1}{n\, h_{n}^{7/2}} + \frac{1}{n^2\, h_{n}^{15/2}} \right) = o_{\P}(1)
\end{equation}
as soon as $n\, h_{n}^{15/4} \rightarrow +\infty$. Now by Cauchy-Schwarz inequality,
\begin{equation}
\label{CvgRnStable}
\frac{\vert R_{n} \vert}{\sqrt{h_{n}}} \leq 2\, \sqrt{\frac{I_{n}\, T_{n}}{h_{n}}} = O_{\P}\bigg( \sqrt{\frac{I_{n}}{h_{n}}} \bigg) = o_{\P}(1)
\end{equation}
from the previous reasoning, the asymptotic normality \eqref{AsNormBR} and Assumption  \ref{hyp:Halpha} with $\alpha=4$. It follows from \eqref{DecompStable}, \eqref{CvgInStable} and \eqref{CvgRnStable} that
\begin{equation*}
\frac{\widehat{T}_{n} - \mu}{\sqrt{h_{n}}} = \frac{T_{n} - \mu}{\sqrt{h_{n}}} + o_{\P}(1)
\end{equation*}
which ends the first part of the proof. The second part makes use of a result of Bickel and Rosenblatt \cite{BickelRosenblatt73}. Indeed, denote by $\bar{T}_{n}$ the statistic given in \eqref{BRStatF0} built on the strong white noise $(\veps_{t})$ instead of the residuals. They show that
\begin{eqnarray}
\label{CvgUnStable}
\frac{\vert T_{n} - \bar{T}_{n} \vert}{\sqrt{h_{n}}} = o_{\P}(1)
\end{eqnarray}
as soon as $n\, h_{n}^{9/2} \rightarrow 0$. A similar calculation leads to
\begin{equation*}
\widetilde{T}_{n} = I_{n} + \bar{T}_{n} + \bar{R}_{n}
\end{equation*}
where
\begin{equation*}
\bar{R}_{n} = 2\, n\, h_{n}\, \int_{\dR} \big( \widehat{f}_{n}(x) - f_{n}(x) \big)\, \big( f_{n}(x) - f(x) \big)\, a(x)\, \dd x.
\end{equation*}
We deduce that
\begin{equation*}
\frac{\vert \widetilde{T}_{n} - \bar{T}_{n} \vert}{\sqrt{h_{n}}} \leq \frac{I_{n}}{\sqrt{h_{n}}} + \frac{\vert \bar{R}_{n} \vert}{\sqrt{h_{n}}} = O_{\P}\bigg( \sqrt{\frac{I_{n}}{h_{n}}} \bigg) = o_{\P}(1)
\end{equation*}
as soon as $\bar{T}_{n}$ satisfies the original Bickel-Rosenblatt convergence and \eqref{CvgInStable} and \eqref{CvgRnStable} hold, that is $n\, h_{n}^{9/2} \rightarrow 0$ and $n\, h_{n}^4 \rightarrow +\infty$. It only remains to note that
\begin{eqnarray*}
\frac{\vert \widehat{T}_{n} - \widetilde{T}_{n} \vert}{\sqrt{h_{n}}} 
%& = & \frac{\vert \widehat{T}_{n} - T_{n} + T_{n} - \bar{T}_{n} + \bar{T}_{n} - \widetilde{T}_{n} \vert}{\sqrt{h_{n}}} \\
 & \leq & \frac{\vert \widehat{T}_{n} - T_{n} \vert}{\sqrt{h_{n}}} + \frac{\vert T_{n} - \bar{T}_{n} \vert}{\sqrt{h_{n}}} + \frac{\vert \bar{T}_{n} - \widetilde{T}_{n} \vert}{\sqrt{h_{n}}},
\end{eqnarray*}
each term being $o_{\P}(1)$.
The proof of Theorem \ref{ThmStable} is now complete.

\subsection{Purely explosive case - Proof of Theorem \ref{ThmExplo}}
\label{SecProofExplo}

Let $I_{n}$ be the first term of the right-hand side of \eqref{DecompStable}, like in the last proof, and note that
\begin{eqnarray}
\label{DecompExplo}
I_{n} ~ = ~ n\, h_{n} \int_{\dR} \big( \widehat{f}_{n}(x) - f_{n}(x) \big)^2 a(x)\, \dd x & \leq & 2\, (I_{n, 1} + I_{n, 2}),
\end{eqnarray}
where for an arbitrary rate that we set to $v_{n}^{\, 2} = \ln(n\, h_{n})$,
\begin{eqnarray*}
I_{n, 1} & = & \frac{1}{n\, h_{n}} \int_{\dR} \left\{ \sum_{t=1}^{n-[v_{n}]} \left(\dK\left(\frac{x-\widehat{\veps}_{t}}{h_{n}} \right) - \dK\left(\frac{x-\veps_{t}}{h_{n}} \right)\right) \right\}^{\! 2} a(x)\, \dd x, \\
I_{n, 2} & = & \frac{1}{n\, h_{n}} \int_{\dR} \left\{ \sum_{t=n-[v_{n}]+1}^{n} \left(\dK\left(\frac{x-\widehat{\veps}_{t}}{h_{n}} \right) - \dK\left(\frac{x-\veps_{t}}{h_{n}} \right) \right) \right\}^{\! 2} a(x)\, \dd x.
\end{eqnarray*}
Under our hypotheses, this choice of $(v_{n})$ ensures
\begin{equation}
\label{PropVn}
\limn v_{n} = +\infty \hsp \text{and} \hsp \limn \frac{v_{n}^{\, 2}}{n \sqrt{h_{n}}} = 0.
\end{equation}
Let us look at $I_{n, 1}$. By Cauchy-Schwarz,
\begin{eqnarray*}
I_{n, 1} & \leq & \frac{n-[v_{n}]}{n\, h_{n}} \int_{\dR} \sum_{t=1}^{n-[v_{n}]} \left[ \dK\left(\frac{x-\widehat{\veps}_{t}}{h_{n}} \right) - \dK\left(\frac{x-\veps_{t}}{h_{n}} \right) \right]^2 a(x)\, \dd x \\
 & \leq & \frac{2\, \Vert \dK \Vert_{\infty}\, (n-[v_{n}])}{n\, h_{n}} \int_{\dR} \sum_{t=1}^{n-[v_{n}]} \left\vert \dK\left(\frac{x-\widehat{\veps}_{t}}{h_{n}} \right) - \dK\left(\frac{x-\veps_{t}}{h_{n}} \right) \right\vert a(x)\, \dd x \\
 & \leq & \frac{2\, \Vert a \Vert_{\infty}\, \Vert \dK \Vert_{\infty}\, (n-[v_{n}])}{n} \int_{\dR} \sum_{t=1}^{n-[v_{n}]} \left\vert \dK\left(u + \frac{\veps_{t} - \widehat{\veps}_{t}}{h_{n}} \right) - \dK(u) \right\vert\, \dd u \\
 & \leq & \frac{2\, B\, \Vert a \Vert_{\infty}\, \Vert \dK \Vert_{\infty}\, (n-[v_{n}])}{n\, h_{n}} \sum_{t=1}^{n-[v_{n}]} \left\vert \veps_{t} - \widehat{\veps}_{t} \right\vert \\
 & = & \frac{2\, B\, \Vert a \Vert_{\infty}\, \Vert \dK \Vert_{\infty}\, (n-[v_{n}])}{n\, h_{n}} \sum_{t=1}^{n-[v_{n}]} \big\vert (\hTTn - \theta)\tr\, \Phi_{t-1} \big\vert \\
 & = & \frac{2\, B\, \Vert a \Vert_{\infty}\, \Vert \dK \Vert_{\infty}\, (n-[v_{n}])}{n\, h_{n}} \sum_{t=1}^{n-[v_{n}]} \big\vert (\hTTn - \theta)\tr\, C_{\theta}^{n}\, C_{\theta}^{-[v_{n}]}\, C_{\theta}^{-n+[v_{n}]}\, \Phi_{t-1} \big\vert. \\
\end{eqnarray*}
Here we recall that $\vert \lambda_{p} \vert > 1$. Consequently, $\rho(C_{\theta}^{-1}) = 1/\vert \lambda_{p} \vert < 1$. It follows that there exists a matrix norm $\Vert \cdot \Vert_{*} = \sup(\vert \cdot\, u \vert_{*}\,;\, u \in \dC^{p},\, \vert u \vert_{*} = 1)$ satisfying $\Vert C_{\theta}^{-1} \Vert_{*} < 1$ (see \textit{e.g.} \cite[Prop. 2.3.15]{Duflo97}), and a constant $k^{*}$ such that, with $C^{*} = 2\, B\, k^{*}\, \Vert a \Vert_{\infty}\, \Vert \dK \Vert_{\infty}$, we obtain
\begin{eqnarray*}
I_{n, 1} & \leq & \frac{C^{*}\, (n-[v_{n}])}{n\, h_{n}} \sum_{t=1}^{n-[v_{n}]} \big\vert (\hTTn - \theta)\tr\, C_{\theta}^{n} \big\vert_{*} ~ \big\vert C_{\theta}^{-[v_{n}]}\, C_{\theta}^{-n+[v_{n}]}\, \Phi_{t-1} \big\vert_{*} \\
 & \leq & \frac{C^{*}\, (n-[v_{n}])}{n\, h_{n}} ~ \big\vert (\hTTn - \theta)\tr\, C_{\theta}^{n} \big\vert_{*} ~ \Vert C_{\theta}^{-1} \Vert_{*}^{[v_{n}]} \sum_{t=1}^{n-[v_{n}]} \big\vert C_{\theta}^{-n+[v_{n}]}\, \Phi_{t-1} \big\vert_{*} \\
 & \leq & \frac{C^{*}\, (n-[v_{n}])}{n\, h_{n}} ~ \big\vert (\hTTn - \theta)\tr\, C_{\theta}^{n} \big\vert_{*} ~ \Vert C_{\theta}^{-1} \Vert_{*}^{[v_{n}]} \sup_{0\, \leq\, t\, \leq\, n-[v_{n}]-1} \big\vert C_{\theta}^{-t}\, \Phi_{t} \big\vert_{*} \sum_{\ell=1}^{n-[v_{n}]} \Vert C_{\theta}^{-1} \Vert_{*}^{\ell}.
\end{eqnarray*}
But Proposition \ref{PropExplo} ensures that $\sup_{t} \vert C_{\theta}^{-t}\, \Phi_{t} \vert_{*}$ is a.s. bounded for a sufficiently large $n$, and we also have
\begin{eqnarray*}
(\hTTn - \theta)\tr\, C_{\theta}^{n} & = & \left( \sum_{t=1}^n \Phi_{t-1}\tr\, \veps_{t} \right) ( C_{\theta}^{-n} )\tr ~ ( C_{\theta}^{n} )\tr \left(\sum_{t=1}^n \Phi_{t-1} \Phi_{t-1}\tr \right)^{\! -1} C_{\theta}^{n} \\
 & = & \sum_{t=1}^n ( C_{\theta}^{-n}\, \Phi_{t-1}\, \veps_{t} )\tr G_{n}^{-1}
\end{eqnarray*}
where
\begin{equation*}
G_{n} = C_{\theta}^{-n} \left(\sum_{t=1}^n \Phi_{t-1} \Phi_{t-1}\tr \right) ( C_{\theta}^{-n} )\tr.
\end{equation*}
We deduce that, for some constant $k^{*}$,
\begin{equation*}
\big\vert (\hTTn - \theta)\tr\, C_{\theta}^{n} \big\vert_{*} \leq k^{*}\, \veps_{n}^{\sharp}\, \Vert G_{n}^{-1} \Vert_{*} \sup_{0\, \leq\, t\, \leq\, n-1} \big\vert (C_{\theta}^{-t}\, \Phi_{t})\tr \big\vert_{*} \sum_{\ell=1}^{n} \Vert C_{\theta}^{-1} \Vert_{*}^{\ell}
\end{equation*}
where \cite[Cor. 1.3.21]{Duflo97} shows that $\veps_{n}^{\sharp} = \sup_{t} \vert \veps_{t} \vert = o(\sqrt{n})$ a.s. under our conditions of moments on $(\veps_{t})$. Proposition \ref{PropExplo} and the fact that $\Vert C_{\theta}^{-1} \Vert_{*} < 1$ lead to
\begin{equation}
\label{DecompExplo1}
I_{n, 1} = o\left( \frac{n - [v_{n}]}{\sqrt{n}\, h_{n}} ~ R^{\, [v_{n}]} \right) \as
\end{equation}
for some $0 < R < 1$. Let us now turn to $I_{n, 2}$ for which the same strategy gives
\begin{eqnarray*}
I_{2,n} & \leq & \frac{[v_{n}]}{n\, h_{n}} ~ \sum_{t=n-[v_{n}]+1}^{n} \int_{\dR} \left[ \dK\left( \frac{x - \widehat{\veps}_{t}}{h_{n}} \right) - \dK\left( \frac{x - \veps_{t}}{h_{n}} \right) \right]^2 a(x)\, \dd x \\
 & \leq & \frac{2\, [v_{n}]}{n\, h_{n}} ~ \sum_{t=n-[v_{n}]+1}^{n} \int_{\dR} \left[ \dK^2\!\left( \frac{x - \widehat{\veps}_{t}}{h_{n}} \right) + \dK^2\!\left( \frac{x - \veps_{t}}{h_{n}} \right) \right] a(x)\, \dd x \\
 & \leq & \frac{4\, \Vert a \Vert_{\infty}\, [v_{n}]}{n} ~ \sum_{t=n-[v_{n}]+1}^{n} \int_{\dR} \dK^2(u)\, \dd u \\
 & = & \frac{4\, \Vert a \Vert_{\infty}\, [v_{n}]^2}{n} ~ \int_{\dR} \dK^2(u)\, \dd u.
\end{eqnarray*}
Thus,
\begin{equation}
\label{DecompExplo2}
I_{n, 2} = O\left( \frac{[v_{n}]^2}{n} \right) \as
\end{equation}
Finally, from \eqref{DecompStable}, \eqref{DecompExplo1} and \eqref{DecompExplo2}, we have
\begin{equation}
\label{DecompExploFin}
I_{n} = O\left( \frac{n - [v_{n}]}{\sqrt{n}\, h_{n}} ~ R^{\, [v_{n}]} + \frac{[v_{n}]^2}{n} \right) = o(\sqrt{h_{n}}) \as 
\end{equation}
for some $0 < R < 1$, as a consequence of the properties of $(v_{n})$ in \eqref{PropVn}. The cross term $R_{n}$ is treated in the same way as in the proof of Theorem \ref{ThmStable}. Indeed, since our choice of $(v_{n})$ also ensures that $I_{n} = o(h_{n})$ a.s., the same reasoning leads to the conclusion. It follows from \eqref{DecompExplo} and \eqref{DecompExploFin} that
\begin{equation*}
\frac{\widehat{T}_{n} - \mu}{\sqrt{h_{n}}} = \frac{T_{n} - \mu}{\sqrt{h_{n}}} + o(1) \as
\end{equation*}
which ends the first part of the proof. The second part merely consists in noting that \eqref{CvgUnStable} still holds.
The proof of Theorem \ref{ThmExplo} is now complete.

\subsection{Unstable case - Proof of Propositions \ref{PropUnstable} and \ref{PropUnivUnstable}}
\label{SecProofUnstable}

In this proof, the notation $\Longrightarrow$ refers to the weak convergence of sequences of random elements in $D([0,1])$, the space of right continuous functions on $[0,1]$ having left-hand limits, equipped with the Skorokhod topology (see Billingsley \cite{Billingsley99}). One can find in Thm. 2.7 of the same reference the statement and the proof of the continous mapping theorem. In Lemmas \ref{LemUnstable1} and \ref{LemUnstable2} below, $(Y_{n})_{n \geq 1}$ is a sequence of random variables satisfying, for some $\delta > 0$,
\begin{equation}
\label{InvPrinc}
\frac{1}{n^{\delta}} \sum_{t\, \in\, \cP_{[n\boldsymbol{\cdot}]}} Y_{t} \cvgd L_{\cP}(\boldsymbol{\cdot}) \hsp \text{and} \hsp \frac{1}{n^{\delta}} \sum_{t=1}^{[n\boldsymbol{\cdot}]} (-1)^{[n\boldsymbol{\cdot}]-t}\, Y_{t} \cvgd \Lambda(\boldsymbol{\cdot})
\end{equation}
where $L_{\cP}$ and $\Lambda$ are random paths in $D([0,1])$ and where $\cP_{[n\boldsymbol{\cdot}]}$ stands either for $\{ 1, \hdots, [n\boldsymbol{\cdot}] \}$ or for the sets of even/odd integers in $\{ 1, \hdots, [n\boldsymbol{\cdot}] \}$. We use the convention $\sum_{\varnothing} = 0$ and the notation $[x]$ to designate the integer part of any $x \geq 0$.

\begin{rem}
The first convergence obviously holds for a strong white noise having a finite variance and $\delta=1/2$, \textit{via} Donsker theorem \cite[Thm. 8.2]{Billingsley99}. In this case, $L_{\cP}$ is a Wiener process and $\cP_{[n\boldsymbol{\cdot}]}$ changes its variance (half as many terms leads to a doubled variance). If the distribution is symmetric, then $\Lambda \equiv L_{\cP}$ for $\cP_{[n\boldsymbol{\cdot}]} = \{ 1, \hdots, [n\boldsymbol{\cdot}] \}$, whereas an invariance principle for stationary processes (see, \textit{e.g.}, \cite[Thm. 1]{DedeckerRio00}) enables to identify $\Lambda$ in the nonsymmetric case.
\end{rem}

\begin{lem}[Univariate unstable case]
\label{LemUnstable1}
Let $X_0 = 0$ and, for $1 \leq t \leq n$, consider
\begin{equation*}
X_{t} = \theta\, X_{t-1} + Y_{t}, \hsp \theta = \pm 1,
\end{equation*}
where $(Y_{t})$ satisfies \eqref{InvPrinc}. Then,
\begin{equation*}
\sum_{t=1}^{n} X_{t} = \left\{
\begin{array}{ll}
O_{\P}(n^{\delta+1}) & \mbox{for } \theta=1 \\
O_{\P}(n^{\delta}) & \mbox{for } \theta=-1
\end{array}
\right.
\end{equation*}
and, in both cases, for $k \in \dN$,
\begin{equation*}
\sum_{t=1}^{n} X_{t}^{\, 2 k} = O_{\P}(n^{2 k \delta+1}).
\end{equation*}
\end{lem}
\begin{proof}
Let $\theta=1$. Then,
\begin{eqnarray*}
\frac{1}{n^{\delta+1}} \sum_{t=1}^{n} X_{t} & = & \frac{1}{n} \sum_{t=1}^{n} \frac{1}{n^{\delta}} \sum_{s=1}^{t} Y_{s} \\
 & = & \sum_{t=1}^{n} \int_{\frac{t}{n}}^{\frac{t+1}{n}} \frac{1}{n^{\delta}} \sum_{s=1}^{[n u]} Y_{s}\, \dd u ~ \cvgl ~ \int_{0}^{1} L(u)\, \dd u
\end{eqnarray*}
from the continuous mapping theorem, where $L \equiv L_{\cP}$ for $\cP_{[n\boldsymbol{\cdot}]} = \{ 1, \hdots, [n\boldsymbol{\cdot}] \}$. Following the same lines,
\begin{eqnarray*}
\frac{1}{n^{2 k \delta+1}} \sum_{t=1}^{n} X_{t}^{\, 2 k} & = & \frac{1}{n} \sum_{t=1}^{n} \left( \frac{1}{n^{\delta}} \sum_{s=1}^{t} Y_{s} \right)^{\! 2 k} \\
 & = & \sum_{t=1}^{n} \int_{\frac{t}{n}}^{\frac{t+1}{n}} \left( \frac{1}{n^{\delta}} \sum_{s=1}^{[n u]} Y_{s} \right)^{\! 2 k} \dd u ~ \cvgl ~ \int_{0}^{1} L^{2 k}(u)\, \dd u.
\end{eqnarray*}
Now let $\theta=-1$. Then,
\begin{equation*}
\sum_{t=1}^{n} X_{t} = \sum_{t\, \in\, \cP_{n}} Y_{t}
\end{equation*}
where $\cP_{n} \subset \{ 1, \hdots, n \}$ is the set of even integers (resp. odd integers) between 1 and $n$, for an even (resp. odd) $n$. Then we conclude to the first result. Finally,
\begin{eqnarray*}
\frac{1}{n^{2 k \delta+1}} \sum_{t=1}^{n} X_{t}^{\, 2 k} & = & \frac{1}{n} \sum_{t=1}^{n} \left( \frac{1}{n^{\delta}} \sum_{s=1}^{t} (-1)^{t-s}\, Y_{s} \right)^{\! 2 k} \\
 & = & \sum_{t=1}^{n} \int_{\frac{t}{n}}^{\frac{t+1}{n}} \left( \frac{1}{n^{\delta}} \sum_{s=1}^{[n u]} (-1)^{[n u]-s}\, Y_{s} \right)^{\! 2 k} \dd u ~ \cvgl ~ \int_{0}^{1} \Lambda^{2 k}(u)\, \dd u.
\end{eqnarray*}
\end{proof}

%\begin{cor}[Seasonal unstable case]
%\label{CorUnstable1}
%Let  $X_{-s+1} = \hdots = X_0 = 0$ and, for $1 \leq t \leq n$, consider
%\begin{equation*}
%X_{t} = X_{t-s} + \veps_{t}
%\end{equation*}
%where $(\veps_{t})$ is a white noise of variance $\sigma^2 > 0$ and $s$ is the season. Then $\vert X_{n} \vert = O_{\P}(\sqrt{n})$,
%\begin{equation*}
%\sum_{t=1}^{n} X_{t} = O_{\P}(n^{3/2}), \hsp \sum_{t=1}^{n} X_{t}^2 = O_{\P}(n^2) \hsp \text{and} \hsp \sum_{t=1}^{n} X_{t}^4 = O_{\P}(n^3).
%\end{equation*}
%\end{cor}
\noindent
\textbf{Proof of Proposition \ref{PropUnstable} -- Seasonal case}.
Let  $X_{-s+1} = \hdots = X_0 = 0$ and $\cS_{k} = \{ \ell \in \{ 1, \hdots, n \},\, \ell[s] = k \}$, for $k \in \{ 0, \hdots, s-1 \}$, the notation $\ell[s]$ refering to the remainder of the Euclidean division of $\ell$ by $s$. Then the subsets $\cS_{k}$ form a partition of $\{ 1, \hdots, n \}$, the path $(X_1, \hdots, X_{n})$ is made of $s$ uncorrelated random walks and Donsker theorem directly gives the $O_{\P}(\sqrt{n})$ rate for $\vert X_{n} \vert$. In addition,
\begin{equation*}
\sum_{t=1}^{n} X_{t}^{\, a} = \sum_{k=0}^{s-1} \sum_{t\, \in\, \cS_{k}} X_{t}^{\, a}
\end{equation*}
so that, for $a \in \{1, 2, 4 \}$, the results follow from Lemma \ref{LemUnstable1}.
\hfill $\square$\\

We now turn to the proof of Proposition \ref{PropUnivUnstable}. Following the idea of \cite{LeeNa02}, we make once again the decomposition
\begin{equation}
\label{DecompUnstable}
\frac{\widehat{T}_{n} - \mu}{\sqrt{h_{n}}} = \frac{I_{n}}{\sqrt{h_{n}}} + \frac{R_{n}}{\sqrt{h_{n}}} + \frac{T_{n} - \mu}{\sqrt{h_{n}}}
\end{equation}
where $\mu$ is the centering term \eqref{CenterBR} of the statistic, $I_{n}$ is given in \eqref{DecompStable2T}, $T_{n}$ is the original Bickel-Rosenblatt statistic and $R_{n}$ is the cross term in \eqref{DecompStable}
such that
\begin{equation}
\label{CvgRnUnstable}
\frac{\vert R_{n} \vert}{\sqrt{h_{n}}} = O_{\P}\bigg( \sqrt{\frac{I_{n}}{h_{n}}} \bigg)
\end{equation}
as we have seen earlier. Going on with the decomposition of $I_{n}$ in \eqref{DecompStable2T} and using the same notation, we establish the two following lemmas. In all the sequel, the case $p=s$ refers to the seasonal process $\theta_1 = \hdots = \theta_{s-1} = 0$ and $\theta_{s}=1$.
\begin{lem}
\label{LemUnstable3}
When $p=1$ or $p=s$, Lemma \ref{LemStable2} still holds in the unstable case.
\end{lem}
\begin{proof}
First, with $p=1$ and $\theta = \pm 1$, we note that
\begin{equation*}
\dE[X_{n}^2] = \sum_{1\, \leq\, t, s\, \leq\, n} \theta^{\, 2n - t-s}\, \dE[\veps_{t}\, \veps_{s}] = O(n)
\end{equation*}
and
\begin{equation*}
\dE[X_{n}^4] = \sum_{1\, \leq\, t, s, u, v\, \leq\, n} \theta^{\, 4n - t-s-u-v}\, \dE[\veps_{t}\, \veps_{s}\, \veps_{u}\, \veps_{v}] = O(n^2)
\end{equation*}
under our assumptions on the noise, and the same results obviously hold for $p=s$ where $X_{n}$ still has a random walk behavior. Thus, $K_{2,n}$ in \eqref{DecompStableK2} is $O_{\P}(n^3\, h_{n})$. Lemma \ref{LemUnstable1} and the second part of Proposition \ref{PropUnstable} also give
\begin{equation*}
\sum_{t=1}^{n} X_{t}^2 = O_{\P}(n^2) \hsp \text{and} \hsp \sum_{t=1}^{n} X_{t}^4 = O_{\P}(n^3)
\end{equation*}
and, accordingly with Proposition \ref{PropOLSUnstable}, $K_{1,n}$ in \eqref{DecompStableK1} is $O_{\P}(n^3\, h_{n}^{-2})$ and $K_{3,n}$ in \eqref{DecompStableK3} is $O_{\P}(n^4\, h_{n}^2)$. Finally, \eqref{DecompStableIn2} concludes the proof.
\end{proof}

\begin{lem}
\label{LemUnstable4}
Let $p=1$. Then, Lemma \ref{LemStable1} holds for $\lambda=\theta=-1$. On the contrary, for $\lambda=\theta=1$ and $\int_{\dR} \dK^{\prime}(s)\, \dd s \neq 0$,
\begin{equation*}
I_{1,n} = O_{\P}(n\, h_{n}^2)
\end{equation*}
whereas for $\int_{\dR} \dK^{\prime}(s)\, \dd s = 0$,
\begin{equation*}
I_{1,n} = O_{\P}(n\, h_{n}^4).
\end{equation*}
The last two results also hold for $p=s$.
\end{lem}
\begin{proof}
On the one hand, let $\theta=-1$. The reasoning of $\eqref{DecompStableJ1}$ shows that $J_{1,n}$ is still $O_{\P}(h_{n})$, using Proposition \ref{PropOLSUnstable} and $\dE[X_{n}^2] = O(n)$. From Lemma \ref{LemUnstable1}, we know that
\begin{equation*}
\sum_{t=1}^{n} X_{t} = O_{\P}(\sqrt{n}).
\end{equation*}
Then, $J_{2,n}$ in \eqref{DecompStableJ2} must be $O_{\P}(n^{-1}\, h_{n}^2)$ and $I_{1,n} = O_{\P}(h_{n})$. On the other hand, we have to investigate the case where $\theta=1$. Since $\dE[X_{n}^2]$ does not depend on $\theta$, $J_{1,n}$ has exactly the same behavior. For a better readability, let
\begin{equation*}
L_{n} = \int_{\dR} \left( \int_{\dR} \dK^{\prime}(z)\, f(x - h_{n}\, z)\, \dd z \right)^{\! 2} a(x)\, \dd x.
\end{equation*}
Under our assumptions, it is easy to see that
\begin{equation}
\label{UnstableLn1}
\limn L_{n} = \left(\int_{\dR} f^2(s)\, a(s)\, \dd s\right) \left( \int_{\dR} \dK^{\prime}(s)\, \dd s\right)^{\! 2}
\end{equation}
as soon as $\int_{\dR} \dK^{\prime}(s)\, \dd s \neq 0$, and that
\begin{equation}
\label{UnstableLn2}
\limn \frac{L_{n}}{h_{n}^{2}} = \left(\int_{\dR} f^{\prime}(s)^2\, a(s)\, \dd s\right) \left( \int_{\dR} s\, \dK^{\prime}(s)\, \dd s\right)^{\! 2}
\end{equation}
as soon as $\int_{\dR} \dK^{\prime}(s)\, \dd s = 0$. Lemma \ref{LemUnstable1} also gives
\begin{equation*}
\sum_{t=1}^{n} X_{t} = O_{\P}(n^{3/2})
\end{equation*}
which shows that $J_{2,n}$ is $O_{\P}(n\, h_{n}^2)$ or $O_{\P}(n\, h_{n}^4)$, depending on $\int_{\dR} \dK^{\prime}(s)\, \dd s$. The latter reasoning still applies for $p=s$, \textit{via} the second part of Proposition \ref{PropUnstable}. The proof is achieved using \eqref{DecompStableIn1} and the hypotheses of Proposition \ref{PropUnivUnstable}.
\end{proof}

The combination of Lemmas \ref{LemUnstable3} and \ref{LemUnstable4} is sufficient to establish that
\begin{equation*}
\frac{\widehat{T}_n - \mu}{\sqrt{h_{n}}} \cvgl \cN(0, \tau^2)
\end{equation*}
holds for $\theta=-1$, despite the instability, and replacing $\widehat{T}_{n}$ by $\widetilde{T}_{n}$ is possible without disturbing the asymptotic normality, as a consequence of \eqref{CvgUnStable} for $n\, h_{n}^{9/2} \rightarrow 0$. When $\theta=1$ or $p=s$ and $\int_{\dR} \dK^{\prime}(s)\, \dd s = 0$, Lemmas \ref{LemUnstable3} and \ref{LemUnstable4} show that $I_{n} = O_{\P}(h_{n}) = o_{\P}(\sqrt{h_{n}})$. It follows from \eqref{DecompUnstable} and \eqref{CvgRnUnstable} that
\begin{equation*}
\frac{\vert R_{n} \vert}{\sqrt{h_{n}}} = O_{\P}(1)
\end{equation*}
which unfortunately prevents us from concluding to the Bickel-Rosenblatt convergence in this case. However, we still have the order of magnitude
\begin{equation*}
\frac{\widehat{T}_n - \mu}{\sqrt{h_{n}}} = O_{\P}(1) \hsp \text{and} \hsp \frac{\widetilde{T}_n - \mu}{\sqrt{h_{n}}} = O_{\P}(1).
\end{equation*}
Finally for $\int_{\dR} \dK^{\prime}(s)\, \dd s \neq 0$, the same lines also imply, together with \eqref{DecompUnstable}, Lemmas \ref{LemUnstable3} and \ref{LemUnstable4} and the asymptotic normality \eqref{AsNormBR}, that
\begin{equation*}
h_{n} (\widehat{T}_{n} - \mu) = O_{\P}(1).
\end{equation*}
It remains to study the asymptotic behavior of the correctly renormalized statistic using the explicit expression of $J_{2,n}$ in \eqref{DecompStableJ2}. From the proof of Lemma \ref{LemUnstable1}, Proposition \ref{PropOLSUnstable} and the continuous mapping theorem, we get for the univariate unstable case with $\theta=1$,
\begin{equation*}
\big( n\, (\hTTn - \theta ) \big)^2 \left( \frac{1}{n^{3/2}} \sum_{t=1}^{n} X_{t-1} \right)^{\! 2} \cvgl \sigma^2 \left( \frac{\frac{1}{2}\, (W^2(1)-1)}{\int_0^1 W^2(u)\, \dd u}\, \int_0^1 W(u)\, \dd u \right)^{\! 2}
\end{equation*}
where $(W(t),\, t \in [0,1])$ is a standard Wiener process. Now for the seasonal case, reusing the same notation as above,
\begin{equation}
\label{SeasUnstable}
\frac{1}{n^{3/2}} \sum_{t=1}^{n} X_{t} = \sum_{k=0}^{s-1} \frac{1}{n^{3/2}} \sum_{t\, \in\, \cS_{k}} X_{t} \cvgl \sigma\, \sum_{k=0}^{s-1} \int_0^\frac{1}{s} W_{s}^{(k+1)}(u)\, \dd u
\end{equation}
where $W_{s}^{(i)}$ stands for the $i$--th component of the standard Wiener process $(W_{s}(t),\, t \in [0,1])$ of dimension $s$, since the path can be seen as the concatenation of $s$ uncorrelated random walks. Using the Cram\`er-Wold and the continuous mapping theorems, it follows that
\begin{equation*}
\left( n\, (\hTTn - \theta )\tr\, \frac{1}{n^{3/2}} \sum_{t=1}^{n} \Phi_{t-1} \right)^{\! 2} \cvgl \sigma^2 \big( S(W_{s})^{T}\, H(W_{s}) \big)^2
\end{equation*}
with the notation of Proposition \ref{PropOLSUnstable}, where $H(W_{s})$ is a random vector of size $s$ containing the limit variable \eqref{SeasUnstable} on each component. The proof is then ended using the limiting behavior of $L_n$. The proof of Theorem \ref{PropUnivUnstable} is now complete.

\begin{rem}
To go beyond, note that for $\theta=1$ and $\int_{\dR} \dK^{\prime}(s)\, \dd s \neq 0$, we get
\begin{equation*}
h_{n}\, \vert \widehat{T}_{n} - \widetilde{T}_{n} \vert \leq 2\, h_{n}\,\sqrt{\widehat{T}_{n}\, U_{n}}  + h_{n}\, U_{n}
\end{equation*}
where
\begin{equation*}
U_{n} = n\, h_{n} \int_{\dR} \big( (\dK_{h_{n}} * f)(x) - f(x) \big)^2\, a(x)\, \dd x.
\end{equation*}
A straightforward calculation shows that $U_{n} = O(n\, h_{n}^3)$ if $\int_{\dR} s\, \dK(s)\, \dd s \neq 0$ whereas $U_{n} = O(n\, h_{n}^5)$ if $\int_{\dR} s\, \dK(s)\, \dd s = 0$. In the second case, $\widehat{T}_{n}$ may be replaced by $\widetilde{T}_{n}$ as soon as $n\, h_{n}^6 \rightarrow 0$. But it the first case, one needs $n\, h_{n}^4 \rightarrow 0$ which contradicts  \ref{hyp:Halpha} with $\alpha=4$.
\end{rem}

This last lemma is not used as part of this paper, but it may be a trail for future studies. It is related to the conclusion of Section \ref{SecBR} and illustrates the phenomenon of compensation in purely unstable processes.

\begin{lem}
\label{LemUnstable2}
Let $X_{-1} = X_0 = 0$ and, for $1 \leq t \leq n$, consider
\begin{equation*}
X_{t} = 2 \cos(\omega)\, X_{t-1} - X_{t-2} + Y_{t}, \hsp 0 \leq \omega \leq \pi,
\end{equation*}
where $(Y_{t})$ satisfies \eqref{InvPrinc}. Then,
\begin{equation*}
\sum_{t=1}^{n} X_{t} = \left\{
\begin{array}{ll}
O_{\P}(n^{\delta+2}) & \mbox{for } \omega=0 \\
O_{\P}(n^{\delta}) & \mbox{for } 0 < \omega \leq \pi.
\end{array}
\right.
\end{equation*}
Moreover, if the process is generated by
\begin{equation*}
X_{t} = X_{t-2} + Y_{t},
\end{equation*}
then 
\begin{equation*}
\sum_{t=1}^{n} X_{t} = O_{\P}(n^{\delta+1}).
\end{equation*}
\end{lem}
\begin{proof}
For $\omega = 0$, let $Z_{t} = X_{t}-X_{t-1}$. Then, $Z_{t}-Z_{t-1} = Y_{t}$ and from Lemma \ref{LemUnstable1},
\begin{equation*}
\sum_{t=1}^{n} Z_{t} = O_{\P}(n^{\delta+1}).
\end{equation*}
A second application of Lemma \ref{LemUnstable1} gives the result. If $0 < \omega \leq \pi$, we obtain
\begin{equation*}
\frac{2\, (1 - \cos(\omega))}{n^{\delta}}\, \sum_{t=1}^{n} X_{t} = \frac{X_{n-1} + (1-2 \cos(\omega))\, X_{n}}{n^{\delta}} + \frac{1}{n^{\delta}} \sum_{t=1}^{n} Y_{t}.
\end{equation*}
In the last case, let $Z_{t} = X_{t}-X_{t-1}$ and note that $Z_{t} = -Z_{t-1} + Y_{t}$. Then, Lemma \ref{LemUnstable1} gives
\begin{equation*}
\sum_{t=1}^{n} Z_{t} = O_{\P}(n^{\delta}).
\end{equation*}
We conclude by applying again Lemma \ref{LemUnstable1} to $X_{t} = X_{t-1} + Z_{t}$.
\end{proof}

\begin{rem}
In the previous lemma, $\omega=0$ corresponds to the unit roots $\{ 1, 1 \}$ in an AR(2) process generated by $(Y_{t})$ whereas $0 < \omega \leq \pi$ corresponds to the unit roots $\{ \de^{\di\, \omega}, \de^{-\di\, \omega} \}$, including $\{-1,-1\}$ for $\omega=\pi$. The last case corresponds to the unit roots $\{ -1,1 \}$, in other words this is the seasonal model for $s=2$.
\end{rem}

\bibliographystyle{plain}
\nocite{*}
\bibliography{BR-AR}

\end{document}